\documentclass{amsart}
\usepackage{amssymb,amsmath,amsthm,amsfonts}
\usepackage {latexsym,enumerate}
\usepackage{bbm}

\allowdisplaybreaks
\newcommand{\nc}{\newcommand}
\nc{\les}{\lesssim}
\nc{\nit}{\noindent}
\nc{\nn}{\nonumber}
\nc{\D}{\partial}
\nc{\diff}[2]{\frac{d #1}{d #2}}
\nc{\diffn}[3]{\frac{d^{#3} #1}{d {#2}^{#3}}}
\nc{\pdiff}[2]{\frac{\partial #1}{\partial #2}}
\nc{\pdiffn}[3]{\frac{\partial^{#3} #1}{\partial{#2}^{#3}}}
\nc{\abs}[1] {\lvert #1 \rvert}
\nc{\cAc}{{\cal A}_c}
\nc{\cE}{{\cal E}}
\nc{\cF}{{\cal F}}
\nc{\cP}{{\cal P}}
\nc{\cV}{{\cal V}}
\nc{\cQ}{{\cal Q}}
\nc{\cGin}{{\cal G}_{\rm in}}
\nc{\cGout}{{\cal G}_{\rm out}}
\nc{\cO}{{\cal O}}
\nc{\Lav}{{\cal L}_{\rm av}}
\nc{\cL}{{\cal L}}
\nc{\cB}{{\cal B}}
\nc{\cZ}{{\cal Z}}
\nc{\cR}{{\cal R}}
\nc{\cT}{{\cal T}}
\nc{\cY}{{\cal Y}}
\nc{\cX}{{\cal X}}
\nc{\cXT}{{{\cal X}(T)}}
\nc{\cBT}{{{\cal B}(T)}}
\nc{\vD}{{\vec \mathcal{D}}}
\nc{\efield}{\mathcal{E}}
\nc{\vE}{{\vec \efield}}
\nc{\vB}{{\vec \mathcal{B}}}
\nc{\vH}{{\vec \mathcal{H}}}
\nc{\mR}{\mathcal R}
\nc{\mF}{\mathcal F}
\nc{\ty}{{\tilde y}}
\nc{\tu}{{\tilde u}}
\nc{\tV}{{\tilde V}}
\nc{\Pc}{{\bf P_c}}
\nc{\bx}{{\bf x}}
\nc{\bX}{{\bf X}}
\nc{\bXYZ}{{\bf XYZ}}
\nc{\bY}{{\bf Y}}
\nc{\bF}{{\bf F}}
\nc{\bS}{{\bf S}}
\nc{\dV}{{\delta V}}
\nc{\dE}{{\delta E}}
\nc{\TT}{{\Theta}}
\nc{\dPsi}{{\delta\Psi}}
\nc{\order}{{\cal O}}
\nc{\Rout}{R_{\rm out}}
\nc{\eplus}{e_+}
\nc{\eminus}{e_-}
\nc{\epm}{e_\pm}
\nc{\sgn}{\text{sgn}}
\nc{\eps}{\varepsilon}
\nc{\vnabla}{{\vec\nabla}}
\nc{\G}{\Gamma}
\nc{\w}{\omega}
\nc{\mh}{h}
\nc{\mg}{g}
\nc{\vphi}{\varphi}
\nc{\tlambda}{\tilde\lambda}
\nc{\be}{\begin{equation}}
\nc{\ee}{\end{equation}}
\nc{\ba}{\begin{eqnarray}}
\nc{\ea}{\end{eqnarray}}

\nc{\g}{\gamma}
\nc{\ol}{\overline}

\newtheorem{theorem}{Theorem}[section]
\newtheorem{lemma}[theorem]{Lemma}
\newtheorem{prop}[theorem]{Proposition}
\newtheorem{corollary}[theorem]{Corollary}

\nc{\pT}{\partial_T}
\nc{\pz}{\partial_z}
\nc{\pt}{\partial_t}
\nc{\la}{\langle}
\nc{\ra}{\rangle}
\nc{\infint}{\int_{-\infty}^{\infty}}
\nc{\halfwidth}{6.5cm}
\nc{\figwidth}{10cm}
\newcommand{\f}{\frac}

\nc{\nlayers}{L} \nc{\nsectors}{M}
\nc{\indicator}{\mathbf{1}}
\nc{\Rhole}{R_{\rm hole}}
\nc{\Rring}{R_{\rm ring}}
\nc{\neff}{n_{\rm eff}}
\nc{\Frem}{F_{\rm rem}}
\nc{\R}{\mathbb R}
\nc{\mJ}{\mathcal J}
\nc{\C}{\mathbb C}
\nc{\Z}{\mathbb Z}
\nc{\N}{\mathbb N}
\nc{\DD}{\Delta}
\nc{\cD}{\mathcal D}
\nc{\lnorm}{\left\|}
\nc{\rnorm}{\right\|}
\nc{\rnormp}{\right\|_{\ell^{p,\eps}}}
\nc{\rar}{\rightarrow} 
\begin{document}
	
	\begin{abstract}
		
		We consider the higher order Schr\"odinger operator $H=(-\Delta)^m+V(x)$ in $n$ dimensions with real-valued potential $V$ when $n>2m$, $m\in \mathbb N$, $m>1$.  When $n$ is odd, we prove that the wave operators  
		extend to bounded operators on $L^p(\mathbb R^n)$ for all $1\leq p\leq\infty$  under $n$ and $m$ dependent conditions  on the potential analogous to the case when $m=1$.  Further, if $V$ is small in certain norms,  that depend  $n$ and $m$, the wave operators are bounded on the same range for even $n$.  We further show that if the smallness assumption is removed in even dimensions the wave operators remain bounded in the range $1<p<\infty$.

	\end{abstract}

	\title[Wave operators for higher order  Schr\"odinger operators]{\textit{The $L^p$-continuity of wave operators for higher order Schr\"odinger operators } } 
	
	\author[M.~B. Erdo\smash{\u{g}}an, W.~R. Green]{M. Burak Erdo\smash{\u{g}}an and William~R. Green}
	\thanks{  The first author is partially supported by Simons Foundation Grant 634269. The second author is partially supported by Simons Foundation
Grant 511825. }
	\address{Department of Mathematics \\
		University of Illinois \\
		Urbana, IL 61801, U.S.A.}
	\email{berdogan@illinois.edu}
	\address{Department of Mathematics\\
		Rose-Hulman Institute of Technology \\
		Terre Haute, IN 47803, U.S.A.}
	\email{green@rose-hulman.edu}

	\maketitle

	\section{Introduction}
	
	We consider the higher order Schr\"odinger equation 
	\begin{align*}
	i\psi_t =(-\Delta)^m\psi +V\psi, \qquad x\in \R^n, \qquad m>1, \quad  m\in \mathbb N.
	\end{align*}
	We restrict our focus to the case when the spatial dimension $n>2m$.  Here $V$ is a real-valued, decaying potential.  We denote the free higher order Schr\"odinger operator by $H_0=(-\Delta)^m$ and the perturbed operator by $H=(-\Delta)^m+V(x)$. We study the $L^p$ boundedness of the wave operators, which are defined by
	$$
	W_{\pm}=s\text{\ --}\lim_{t\to \pm \infty} e^{itH}e^{-itH_0}.
	$$
	For the classes of potentials $V$ we consider, the wave operators exist and are asymptotically complete, see the work of Agmon, \cite{agmon}, H\"ormander, \cite{Hor} and  Schechter, \cite{ScheArb,Sche}.
	
	We use the notation $\la x\ra$ to denote $  (1+|x|^2)^{\f12}$, $\mathcal F(f)$ or $\widehat f$ to denote the Fourier transform of $f$.  We write $A\les B$ to say that there exists a constant $C$ with $A\leq C B$, and write $a-:=a-\epsilon$ and $a+:=a+\epsilon$ for some $\epsilon>0$ throughout the paper.  We use the norm $\|f\|_{H^{\delta}}=\|\la \cdot \ra^{\delta} \widehat f(\cdot)\|_2$.   We first state a small potential result that is valid in all dimensions $n>2m$.

	\begin{theorem}\label{thm:small}
		Let $n>2m$. 
		Assume that the $V$ is a real-valued potential on $\R^n$ and   fix $0<\delta\ll 1$. Then $\exists C=C(\delta,n,m)>0 $ so that the wave operators  extend   to   bounded operators on $L^p(\R^n)$ for all $1\leq p\leq \infty$, provided that 
		\begin{enumerate}[i)]
			
			\item $\big\| \la \cdot \ra^{\frac{4m+1-n}{2}+\delta} V(\cdot)\big\|_{2}<C$   when $2m<n<4m-1$,
			
			\item $\big\|\la \cdot \ra^{1+\delta}V(\cdot)\big\|_{H^{\delta}}<C$ when $n=4m-1$,
			
			\item  $\big\|\mathcal F(\la \cdot \ra^{\sigma} V(\cdot))\big\|_{L^{ \frac{n-1-\delta}{n-2m-\delta} }}<C$ for some $\sigma>\frac{2n-4m}{n-1-\delta}+\delta$   when $n>4m-1$.
			
		\end{enumerate}
		
	\end{theorem}
	
For boundedness on $L^p$ when $1<p<\infty$, we may remove the smallness assumption above provided $V$ decays sufficiently at spatial infinity.  We define zero energy to be regular if there are no non-trivial distributional solutions to $H\psi=0$ with $\la x \ra^{\frac{n}{2}-2m-}\psi(x) \in L^2$.  We show
	
	\begin{theorem}\label{thm:main}
		Let $n>2m$. 
		Assume that the $V$ is a real-valued potential on $\R^n$ so that
		\begin{enumerate}[i)]
			\item $|V(x)|\les \la x\ra^{-\beta}$ for some $\beta>n+3$ when  $n$ is odd and for some  $\beta>n+4$ when $n$ is even
			
			\item $\|\la \cdot \ra^{1+}V(\cdot)\|_{H^{0+}}<\infty$   when $n=4m-1$,
			
			\item for some $0<\delta\ll 1$ and $\sigma>\frac{2n-4m}{n-1-\delta}$, $\|\mathcal F(\la \cdot \ra^{\sigma} V(\cdot))\|_{L^{ \frac{n-1-\delta}{n-2m-\delta} }}<\infty$   when $n>4m-1$,
			
			\item $H=(-\Delta)^m+V(x)$ has no positive
			eigenvalues and zero energy is regular. 
		\end{enumerate}
	Then,	the wave operators extend to bounded operators on $L^p(\R^n)$ for all $1<p<\infty$.  
 \end{theorem}
Finally, with slightly more decay on the potential we recover the endpoints $p=1,\infty$ in odd dimensions:
\begin{theorem}\label{thm:main2}
		Let $n>2m$ be odd. Assume that $V$ satisfies the hypothesis of Theorem~\ref{thm:main} and in addition  $|V(x)|\les \la x\ra^{-\beta}$ for some $\beta>n+5$.  
	Then,	the wave operators extend to bounded operators on $L^p(\R^n)$ for all $1\leq p\leq \infty$.  
 \end{theorem}

	In even dimensions, we lose the boundedness on the endpoints of $p=1,\infty$ due to the low energy.  In particular, the energies away from zero are bounded on the full range including $p=1,\infty$, see Proposition~\ref{prop:hi even} below.  We hope to address the cases of $p=1,\infty$ when $n>2m$ even and the case when there are threshold obstructions  in a  future work.

	We note that the norm used when $n>4m-1$ is finite when $\la x\ra^{\sigma}V(x)$ has more than $\frac{n}{n-2m}(\frac{n-4m+1}{2})$ derivatives in $L^2(\R^n)$.  In all cases above, we also note that
	$$
		\| V\|_{L^2(B(x,1))}\les \la x\ra^{-1-}, \quad x\in \R^n.
	$$
	This suffices to imply, \cite{Sche,agmon,ScheArb}, the existence, asymptotic completeness, and intertwining identity for the wave operators.
	In particular, we have
	\begin{align}\label{eq:intertwining}
	f(H)P_{ac}(H)=W_\pm f((-\Delta)^m)W_{\pm}^*.
	\end{align}
	Here $P_{ac}(H)$ is the projection onto the absolutely continuous spectral subspace of $H$, and $f$ is any Borel function.  Using \eqref{eq:intertwining} one may obtain $L^p$-based mapping properties for the more complicated, perturbed operator $f(H)P_{ac}(H)$ from the simpler free operator $f((-\Delta)^m)$.  The boundedness of the wave operators on $L^p(\R^n)$ for any choice of $p\geq 2$ with the function $f(\cdot)=e^{-it(\cdot)}$ yield the dispersive estimate
	\begin{align}\label{eqn:disp est}
	\|e^{-itH}P_{ac}(H)\|_{L^{p'}\to L^p}\les |t|^{-\frac{n}{2m}+\frac{n}{pm} },
	\end{align}
	where $p'$ is the H\"older conjugate of $p$.    In particular in all odd dimensions $n>2m$, under the hypothesis of Theorem~\ref{thm:main2}, we have 
$$
	\|e^{-itH}P_{ac}(H)\|_{L^{1}\to L^\infty}\les |t|^{-\frac{n}{2m} }. 
$$	 
	Our work is inspired by recent work by Feng, Soffer, Wu and Yao on weighted $L^2$-based ``local dispersive estimates" for higher order Schr\"odinger operators considered in \cite{soffernew}, as well as the recent work on the $L^p(\R^3)$ boundedness of the wave operators for the fourth order ($m=2$) Schr\"odinger operators by Goldberg and the second author \cite{GG4wave}, and the extensive works of Yajima, \cite{YajWkp1,YajWkp2,YajWkp3, Yaj,YajNew},  in the case of $m=1$.  The wave operators for the usual Schr\"odinger operator $-\Delta+V$, when $m=1$ are well-studied, see for example \cite{YajWkp1,YajWkp2,YajWkp3,JY2,JY4,DF,Miz} in all dimensions $n\geq 1$.  On $\mathbb R^3$, Beceanu and Schlag  obtained detailed structure formulas for the wave operators, \cite{Bec,BS,BS2}.  The $L^2$ existence and other properties of the higher order wave operators have been studied by many authors, including Agmon \cite{agmon}, Kuroda \cite{Kur1, Kur2}, H\"ormander \cite{Hor}, and Schechter, \cite{Sche,ScheArb}.  We note that the only\footnote{During the review period of this article Mizutani, Wan and Yao proved results for the case of $m=2$ and $n=1$, \cite{MWY}.} result on the $L^p$ boundedness of the wave operators for higher order Schr\"odinger operators is the case of $m=2$ and $n=3$ by Goldberg and the second author, \cite{GG4wave}.   There appears to be three regimes in the analysis of $L^p$ boundedness of the wave operators: $n<2m$, $n=2m$, and $n>2m$. In the case $n <2m$, as in \cite{GG4wave},  zero energy is not regular for the free operator and the main difficulty in the analysis is the small energies.  However the large energy argument is more straightforward since the resolvent  decays in the spectral parameter $\lambda$. 
	In the range $n>2m$ the zero energy is regular for the free operator and the resolvent remains bounded as $\lambda\to 0$. However, the large energies, and in particular the Born series terms, are not easy to deal with. When $n> 4m-1$ one needs a smoothness requirement on the potential $V$ as in the case $m=1$ and $n>3$, \cite{YajWkp1,GV}, due to the growth of the resolvents as the spectral variable goes to infinity.  The case $n=2m$ is challenging in both the low and high energy regimes.

	Similar to the usual second order Schr\"odinger operator, for the types of potentials we consider there is a Weyl criterion and $\sigma_{ac}(H)=\sigma_{ac}(H_0)=[0,\infty)$.  In contrast, decay of the potential is not sufficient to ensure the lack of eigenvalues embedded in the continuous spectrum for the higher order operators, \cite{soffernew}.  Even perturbing with compactly supported, smooth potentials may induce embedded eigenvalues.  We leave this as an overarching assumption and note that there are conditions that ensure the lack of embedded eigenvalues, see Theorem~1.11 in \cite{soffernew}.  
	
	To prove Theorem~\ref{thm:main} we use a time-independent representation of the wave operators based on resolvent operators.  We have the splitting identity for $z\in\C\setminus[0,\infty)$, (c.f. \cite{soffernew})
	\be\label{eqn:Resol}
	\mR_0(z)(x,y):=((-\Delta)^m -z)^{-1}(x,y)=\frac{1}{ mz^{1-\frac1m} }
	\sum_{\ell=0}^{m-1} \omega_\ell R_0 ( \omega_\ell z^{\frac1m})(x,y)
	\ee
	where $\omega_\ell=\exp(i2\pi \ell/m)$ are the $m^{th}$ roots of unity, $R_0(z)=(-\Delta-z)^{-1}$ is the usual ($2^{nd}$ order) Schr\"odinger resolvent.  
	Using the change of variables $z= \lambda^{2m}$ with $\lambda$ restricted to the sector in the complex plane with $0<\arg(\lambda)<\pi/m$,
	\be\label{eqn:Resolv}
	\mR_0( \lambda^{2m})(x,y):=((-\Delta)^m -\lambda^{2m})^{-1}(x,y)=\frac{1}{ m\lambda^{2m-2}}
	\sum_{\ell=0}^{m-1} \omega_\ell R_0 ( \omega_\ell  \lambda^2)(x,y).
	\ee
	By the well-known Bessel function expansions, for $n>3$ odd we have
	\be\label{eqn:R0 explicit0}
	R_0(z^2)(x,y)=\frac{e^{iz |x-y|}}{|x-y|^{n-2}} \sum_{j=0}^{\frac{n-3}{2}} c_{n,j} |x-y|^j z^{j}, \qquad \Im(z)>0.
	\ee   
	Even dimensions are more complicated due to the appearance of logarithmic terms.
	
	Our usual starting point to study the wave operators is the stationary representation 
	\begin{align*}
	W_+u
	&=u-\frac{1}{2\pi i} \int_{0}^\infty \mR_V^+(\lambda) V [\mR_0^+(\lambda)-\mR_0^-(\lambda)] u \, d\lambda,
	\end{align*}
	where $\mR_V(\lambda)=((-\Delta)^{m}+V-\lambda)^{-1}$, where the `+' and `-' denote the usual limiting values as $\lambda$ approaches the positive real line from above and below, \cite{soffernew}.
	Since the identity operator is bounded on $L^p$, we need only bound the second term involving the integral.  It is convenient to make the change of variables $\lambda \mapsto\lambda^{2m}$ and consider the integral kernel of the operator
	\begin{align}\label{eqn:wave op defn}
	-\frac{m}{\pi i} \int_0^\infty \lambda^{2m-1}\mR_V^+(\lambda^{2m})V[\mR_0^+-\mR_0^-](\lambda^{2m})\, d\lambda.
	\end{align}
	Our result in Theorem~\ref{thm:small} follows by using resolvent identities to expand $\mR_V^+$ in an infinite series 
	and directly summing the series.  To remove the smallness assumption to show that the operator defined in \eqref{eqn:wave op defn} extends to a bounded operator on $L^p$ requires different strategies in the low ($0<\lambda\ll 1$) and high ($\lambda \gtrsim 1$) energy regimes.  To delineate these cases, we use the even, smooth cut-off function $\chi$ with $\chi(\lambda)=1$ for $|\lambda|<\lambda_0$ for some sufficiently small $\lambda_0\ll 1$, and $\chi(\lambda)=0$ for $|\lambda|>2\lambda_0$, as well as the complimentary cut-off $\widetilde \chi(\lambda)=1-\chi(\lambda)$.
	
	We note that the different assumptions on the potential we impose based on the size of $n$ versus $m$ are natural.  When $n\leq 2m$ the low energy expansions of the resolvent $\mR_0$ are singular as the spectral parameter $\lambda \to 0$.  This complication necessitates a different strategy to invert certain operators and develop expansions for both the free and perturbed resolvents, see \cite{GT4,egt} for the case when $m=2$ and $n=4,3$ respectively.    
	Smoothness of the potential is required for the second order Schr\"odinger  operator  in dimensions $n>3$ since the kernel free resolvent $R_0^\pm(\lambda^2)$ grows like $\lambda^{\frac{n-3}{2}}$ as the spectral parameter $\lambda \to \infty$.  This causes the $L^1\to L^\infty$ dispersive estimates to fail in dimensions greater than three without some smoothness assumptions on the potential, see the counterexample constructed by Goldberg and Visan \cite{GV}.     The higher order Schr\"odinger resolvent,  $\mR_0(\lambda^{2m})$ grows like $\lambda^{\frac{n+1}{2}-2m}$ when $n> 4m-1$, which necessitates a control over derivatives of the potential which we measure in terms of the $\mathcal F L^r$ norm similar to the conditions for the second order Schr\"odinger established by Yajima, \cite{YajWkp1}. Our $\epsilon$-smoothness requirement in the case $n=4m-1$ could be an artifact of our methods.  
	
	We assume that zero energy is regular, that is there are no threshold resonances or eigenvalues.  These can be characterized in terms of distributional solutions to $H\psi=0$, with $\psi$ in weighted $L^2(\R^n)$ spaces, see section 8 of \cite{soffernew}.  The effect of zero energy resonances or eigenvalues on the $L^p$-boundedness of the wave operators is well-studied in for $m=1$ Schr\"odinger operator.  Generically, one sees the range shrink to $1<p<\frac{n}{2}$ when $n\geq 3$, while further orthogonality conditions allows one to obtain a larger range.  See, for example the work of Yajima \cite{YajNew,YajNew2,YajNew3}, also Goldberg and the second author \cite{GGwaveop}.  In the higher order case, one would expect the wave operators to be bounded for $1<p<\frac{n}{2m}$ in the presence of zero energy eigenvalues when $n>4m$, with a larger upper bound on the range of $p$ when $2m<n\leq 4m$ or in the case of resonances, or sufficient cancellation properties between the potential and zero energy eigenspace.  In these cases only the bounds on the low energy portion of the tail of the Born series would be affected.  The effect of embedded eigenvalues has no analogue in the $m=1$ case, its effect on the $L^p$-boundedness of the wave operators is unknown.

	The paper is organized as follows.  We first control the Born series terms that arise by iterating the resolvent identity for the perturbed resolvent in the stationary representation, \eqref{eqn:wave op defn}, of the wave operator in Section~\ref{sec:Born}.   Next, we  prove Theorem~\ref{thm:main} and Theorem~\ref{thm:main2}. First in odd dimensions, in Section~\ref{sec:low} and Section~\ref{sec:low2}, we control the remainder in the low energy regime, when the spectral parameter $\lambda$ is in a neighborhood of zero.  In Section~\ref{sec:high} we control the remainder in the high energy regime, when $\lambda\gtrsim 1$ in odd dimensions.  In Section~\ref{sec:even} we show how the arguments in Sections~\ref{sec:low} and \ref{sec:high} may be adapted to the even dimensional case.  Finally, in Section~\ref{sec:int ests} we provide integral estimates that are used throughout the paper.

	\section{Born Series}\label{sec:Born}

	By iterating the resolvent identity, one has the expansion
	\begin{align}\label{eqn:born identity}
	\mR_V(z)=\sum_{J=0}^{2\ell}\big[ \mR_0(z)(-V\mR_0(z))^J \big]- (\mR_0(z)V)^\ell \mR_V(z) (V\mR_0(z))^\ell.
	\end{align}
	Consider the contribution of an arbitrary summand in the Born series to \eqref{eqn:wave op defn},
	$$
	W_J:=(-1)^{J+1}\frac{1}{2\pi i}\int_0^\infty   (\mR_0^+(\lambda)V)^{J} [\mR_0^+(\lambda)-\mR_0^-(\lambda)]\, d\lambda.
	$$
	In this section by modifying the proof of Yajima  in \cite{YajWkp1} to control the Born series terms for the second order Schr\"odinger, we prove that $W_J$
	extends to a bounded operator on $L^p(\R^n)$, $1\leq p\leq \infty$:
	\begin{theorem}\label{thm:Born}   Fix $1\leq p\leq \infty$ and $0<\delta\ll 1$. Then $\exists C=C(\delta,n,m)>0$ so that
		for $2m<n<4m-1$,  we have  
		$$\|W_J\|_{L^p\to L^p}\leq C^J    \| \la \cdot \ra^{ \frac{4m+1-n}{2}+\delta} V(\cdot )\|_{L^2}^J, $$
		for $n=4m-1$, we have 
		$$\|W_J\|_{L^p\to L^p}\leq C^J  \|  \la x\ra^{1+\delta} V\|_{H^{ \delta}}^{J}, $$
		for $n>4m-1$, we have 
		$$\|W_J\|_{L^p\to L^p}\leq C^J  \| \mF(\la x\ra^{ \frac{2n-4m}{n-1-\delta}+\delta} V)\|_{L^{\frac{n-1-\delta}{n-2m-\delta}}}^J.$$
	\end{theorem}	
	In what follows we will ignore most implicit constants; their affect on the final inequality is of the form $C^J$, where $C$ depends on $n,m$ and the actual value of the implicit small constants in the hypothesis above.  Theorem~\ref{thm:small} follows from this result.
	
	Our approach is inspired by the paper \cite{YajWkp1}, in which Yajima proved the result in the case of $m=1$. We will bound the adjoint operator $Z_J=W_J^*$.  Fix $f\in \mathcal S$  and let  
	\be\label{eqn:ZJ defn}
	Z_Jf(x)=\lim_{ \epsilon_1 \to 0^+}\cdots \lim_{\epsilon_J\to 0^+} \lim_{\epsilon_0\to 0^+} Z_{J,\vec \epsilon,\epsilon_0} f(x),
	\ee
	where 
	$$
	Z_{J,\vec \epsilon,\epsilon_0}f (x)
	:= \frac1{2\pi i} \int_\R \big[\mR_0(\lambda-i\epsilon_0) V\mR_0 (\lambda+i\epsilon_1) \cdots V\mR_0 (\lambda+i\epsilon_J) f\big](x) d\lambda.
	$$
	The main result of this sections is to show this operator is bounded on $L^p(\R^n)$ for all $1\leq p\leq \infty$. As in \cite{YajWkp1}, it suffices to prove that the limit above exists in $L^p$ and the bounds stated in the theorem hold  for $f\in \mathcal S$ and $\widehat V\in C_0^\infty$. 
	
	Taking the Fourier transform in $x$ yields, up to constants,
	\begin{align*}
	\mathcal  F(Z_{J,\vec \epsilon,\epsilon_0}f)(\xi)= \int_\R
	\int_{ \R^{Jn} }\frac{1}{\xi^{2m}-\lambda+i \epsilon_0}\prod_{j=1}^J \frac{\widehat V(k_j)}{(\xi-\sum_{\ell=1}^j k_\ell)^{2m}-\lambda -i\epsilon_j}\widehat f(\xi-\sum_{j=1}^J k_j)
	d\lambda.
	\end{align*}
	Applying Cauchy's integral formula to the $\lambda$ integral in the definition of $Z_J$  and taking $\epsilon_0\to 0^+$ yield
	$$
	\mathcal  F(Z_{J,\vec \epsilon}f)(\xi)=   \int_{\R^{Jn}}  \bigg[\prod_{j=1}^J \frac{\widehat V(k_j)} {  (|\xi-\sum_{\ell=1}^j k_\ell|^{2m}-|\xi|^{2m} - i\epsilon_j )} \bigg]\widehat f (\xi-\sum_{j=1}^J k_j) dk_1 \cdots d k_J.
	$$
	Now, we utilize the change of variables $\sum_{\ell=1}^j k_\ell \mapsto  k_j $ for $j=1,\ldots,J$ and define $k_0=0$ to obtain
	$$
	\mathcal  F(Z_{J,\vec \epsilon}f)(\xi)= \int_{\R^{Jn}}  \bigg[\prod_{j=1}^J \frac{\widehat V(k_j-k_{j-1})} {  (|\xi-k_j|^{2m}-|\xi|^{2m} -i\epsilon_j )} \bigg]\widehat f (\xi- k_J) dk_1 \cdots d k_J.
	$$
	We define the multiplier operator  $T_{k,\epsilon}^m$ by
	\be \label{Tmke}
	T_{k,\epsilon}^mf  =\mathcal F^{-1}\bigg( \frac{\widehat f(\xi)}{|\xi-k|^{2m}-|\xi|^{2m}-i\epsilon} \bigg). 
	\ee
	Let $  K_J(k_1,k_2,\dots, k_J)= \prod_{j=1}^J \widehat V(k_j-k_{j-1})$ and $f_{k_J}(x)=e^{ik_J\cdot x} f(x)$.  Then, we have 
	\begin{multline}\label{eq:ZJfin}
	Z_Jf(x) \\=\lim_{\epsilon_1\to 0^+} \cdots \lim_{\epsilon_J\to 0^+}  
	\int_{ \R^{n} } T_{k_1,\epsilon_1}^m\bigg\{  \int_{ \R^{n} } T_{k_2,\epsilon_2}^m\bigg\{ \cdots \int_{\R^n}  K_J(k_1,k_2,\dots, k_J) T_{k_J, \epsilon_J }^m f_{k_J} \, dk_J \bigg\}\cdots   \bigg\} dk_2 \bigg\} dk_1,
	\end{multline}
	Now, we need to study the operators $T^m_{k,\epsilon}$ in some detail.  We note the algebraic identity
	\begin{multline*}
	|\xi-k|^{2m}-|\xi|^{2m}-i\epsilon=(|\xi-k|^2-|\xi|^2)\big(\sum_{\ell=0}^{m-1}| \xi-k|^{2\ell} |\xi|^{2m-2-2\ell}\big)-i\epsilon\\
	= 2i \frac{|k|^{2m-1}}{p_\omega(\xi/|k|)}\big(-\frac{i|k|}2+i\omega\cdot\xi-\frac{\epsilon p_\omega(\xi/|k|) }{2  |k|^{2m-1} }\big),  
	\end{multline*}
	where
	\be \label{eqn:pomega def}
	\omega=\frac{k}{|k|}\in S^{n-1},\,\,\,\text{ and }\,\,  
	p_\omega(\xi)=\frac1{ \sum_{\ell=0}^{m-1} | \omega- \xi  |^{2\ell} | \xi |^{2m-2-2\ell}}.
	\ee
	We therefore have 
	$$
	T_{k,\epsilon}^mf  
	=\frac1{2i|k|^{2m-1}} \mathcal F^{-1}\bigg( \frac{p_\omega(\xi/|k|)\widehat f(\xi)}{ -\frac{i|k|}2+i\omega\cdot\xi-\frac{\epsilon p_\omega(\xi/|k|)}{2|k|^{2m-1} } } \bigg).
	$$
	Writing (note that $p_\omega(\xi)>0$)
	$$
	\frac{1}{ -\frac{i|k|}2+i\omega\cdot\xi-\frac{\epsilon p_\omega(\xi/|k|)}{2|k|^{2m-1} } }=-\int_0^\infty e^{-\frac{i|k|t}2+ it\omega\cdot\xi} e^{-\frac{\epsilon p_\omega(\xi/|k|)}{2|k|^{2m-1} }t} dt,
	$$	
	we obtain
	$$
	\mathcal F^{-1}\bigg( \frac{p_\omega(\xi/|k|) \widehat f(\xi)}{ -\frac{i|k|}2+i\omega\cdot\xi-\frac{\epsilon p_\omega(\xi/|k|)}{2|k|^{2m-1} }} \bigg)(x)=-\int_0^\infty e^{ -\frac{i|k|t}2 } h_{k,\frac{\epsilon t}{2|k|^{2m-1}}} * f (x+t\omega) dt,
	$$
	where $*$ denotes convolution and
	$$
	h_{k,\epsilon  } =\mF^{-1}\Big(p_\omega(\xi/|k|) e^{-  \epsilon   p_\omega(\xi/|k|)  }\Big).
	$$
	\begin{lemma}\label{lem:gk}
		
		We have the following bounds  (with $k=s\omega, s>0, \omega\in S^{n-1}$)
		$$
		\big\|  \sup_{\epsilon>0} h_{k,\epsilon } \big\|_{L^1}\les 1,
		$$
		$$
		\big\|\sup_{\epsilon>0} |\partial_s^j h_{s\omega,\frac{\epsilon }{ s^{2m-1}}}| \big\|_{L^1}\les  s^{-j},\,\,\, j= 1,2,\ldots 
		$$
		Furthermore, $h_{k,\epsilon}$ converges to $h_k:=h_{k,0}$ and $\partial_s^j h_{s\omega,\frac{\epsilon }{ s^{2m-1}}} $  converges to $\partial_s^j h_k  $ as $\epsilon\to 0$ a.e. and in $L^1$, and $h_k$ satisfies the same bounds above.    
	\end{lemma} 
	\begin{proof} We first prove the claims for $h_k$.
		Note that 
		$$
		\| h_{s\omega}\|_{L^1} = \Big\| \mF^{-1} ( p_{\omega} )\Big\|_{L^1}.
		$$
		A simple calculation shows that
		$$
		\big|\nabla_\xi^N  p_\omega (\xi) \big| \les \frac{1}{\la \xi \ra^{2m-2+N}}, \qquad N=0,1,2,\dots
		$$
		This is seen by considering cases based on the size of $|\xi|$ and $|\omega|=1$ in \eqref{eqn:pomega def}.  Therefore, for $N\geq n-2m+3$, $|x|^N  \mF^{-1} ( p_{\omega} )(x)$ is a bounded continuous function, and hence  
		$$
		\mF^{-1} ( p_{\omega} )(x)=u+O\big(\min(|x|^{-n-1},|x|^{-n+1})\big),
		$$
		where $u$ is a distribution supported at $0$. Since $ p_{\omega}( \xi )  \to 0$ as $|\xi|\to \infty$, we conclude that $u=0$, which yields the claim for $j=0$.  For $j>0$, note that 
		$$
		\partial_s\mF^{-1}p_\omega(sx)=x\cdot [\nabla\mF^{-1} p_{\omega}](xs)=\frac1s \mF^{-1}(\nabla\cdot\xi\, p_{\omega}(\xi)) (xs).
		$$
		Similarly, $\partial_s^\ell \mF^{-1}p_\omega(sx)= s^{-\ell} \mF^{-1}((\nabla\cdot\xi)^\ell p_{\omega}(\xi)) (xs)$. Therefore,  
		$$
		|\partial_s^jh_{s\omega}(x)|\les \sum_{\ell=0}^j s^{n+\ell-j} s^{-\ell}  | \mF^{-1}((\nabla\cdot\xi)^\ell p_{\omega}(\xi)) (xs)|.
		$$
		The claim follows from this as above since 	$(\nabla\cdot\xi)^\ell p_{\omega}(\xi)$ satisfies the same bounds as $p_{\omega}(\xi)$.

		Now, we consider $h_{k,\epsilon}$.  Let $H_\omega(\epsilon,x)=\mF^{-1}\Big(p_\omega e^{-  \epsilon   p_\omega  }\Big)(x)$. 	Using the bounds on the derivatives of $p_\omega$, and noting that  
		$ p_\omega (\xi)\approx \la \xi\ra^{2-2m}$ and that  $\sup_{\alpha>0} \alpha e^{-\alpha}\les 1$, we conclude that
		$$
		\big|\nabla_\xi^N  [p_\omega (\xi) e^{-  \epsilon   p_\omega(\xi ) }]\big| \les \frac{1}{\la \xi \ra^{2m-2+N}}, \qquad N=0,1,2,\dots
		$$
		Therefore we have  
		\be\label{eqn:homega bound}
		|\mF^{-1} ( p_\omega e^{-  \epsilon   p_\omega  } )(x)|\les    \min(|x|^{-n-1},|x|^{-n+1}) ,
		\ee
		uniformly in  $\epsilon>0$.
		This yields the claim for $j=0$ since $h_{k,\epsilon}=s^nH_\omega(\epsilon,sx)$.   
		
		Similarly, note that 
		$$\big|\nabla_\xi^N  [p_\omega (\xi) (e^{-  \epsilon   p_\omega(\xi ) }-1) ]\big|\les \frac{\epsilon }{\la \xi \ra^{4m-4+N}}, \qquad N=0,1,2,\dots.
		$$
		This implies the a.e. and $L^1$ convergence of $h_{k,\epsilon}$  to $h_k $.
		
		For the $j$th derivative of $h_{k,\epsilon}$, by chain rule and scaling as above, it suffices to prove that the $L^1$ norms of $\sup_{\epsilon} \epsilon^{j_1}\partial_\epsilon^{j_1}  (x\cdot\nabla_x)^{j_2} \mF^{-1}[p_{\omega}e^{-\epsilon p_\omega}](x)$ are $\les 1 $ for $j_1,j_2\geq 0$. Note that $$\nabla_\xi^{N}\epsilon^{j_1}\partial_\epsilon^{j_1}  (\nabla_\xi\cdot\xi)^{j_2} p_{\omega}(\xi)e^{-\epsilon p_\omega(\xi)} \in L^1$$
		for $N  \geq n-2m+3$. The claim follows  as above.  Convergence of the $s$ derivatives of  $h_{k,\epsilon}$ follow similarly. 
	\end{proof}
	We conclude that for $f\in\mathcal S$
	$$
	T_{k,\epsilon}^m f(x) =\frac{i}{2|k|^{2m-1}}\int_0^{\infty} \int_{ \R^{ n} } e^{-i|k|t/2} h_{k,\frac{\epsilon t}{2|k|^{2m-1}}}(y ) f(x-y +t\omega)  \, dy  \, dt,
	$$
	and for all $x\in \R^n$
	$$	\lim_{\epsilon\to 0^+} T_{k,\epsilon}^mf(x)= \frac{i}{2|k|^{2m-1}} \int_0^\infty e^{-it|k|/2} \int_{\R^n}h_k(y )f(x-y +t\omega) \, dy  \, dt:= T_k^mf(x). $$
	Following the notation of \cite{YajWkp1}, for $\epsilon>0$, let 
	$$
	G_\epsilon f=\int_{\R^n} T_{k,\epsilon}^m f(k,\cdot) dk,\qquad   G_0 f  =\int_{\R^n} T_{k}^m f (k,\cdot) dk,
	$$
	Note that
	\be\label{eq:geps}
	G_\epsilon f(x)=\int_{\R^n}\frac{i}{2|k|^{2m-1}}\int_0^{\infty} \int_{ \R^{ n} } e^{-i|k|t/2}  h_{k,\frac{\epsilon t}{2|k|^{2m-1}}}(y ) f(k,x-y +t\omega)  \, dy  \, dt\,dk.
	\ee
	Passing to polar coordinates, $k=s\omega$, and changing the order of integration, we have 
	$$
	G_\epsilon f(x)=\frac{i}2\int_{S^{n-1}} \int_0^{\infty} F_\epsilon( t,\omega,x)  \, dt\,d\omega, 
	$$
	where
	$$
	F_\epsilon( t,\omega,x)= \int_0^\infty  e^{-is t/2} s^{n-2m}  h_{s\omega,\frac{\epsilon t}{2s^{2m-1}}}* f(s\omega, \cdot)(x +t\omega ) \, ds.
	$$
	Also note that $G_0f$   satisfies the same formula  with  $F_0 $  replacing  $F_\epsilon $.

	\begin{lemma}\label{lem:Geps} 
		Let  $\epsilon>0$ and  $f(k,x)\in \mathcal S(\R^n_k,\mathcal S(\R^n_x))$. For all  $n>2m+1$ and $1\leq p \leq \infty$, we have
		$$ \| G_\epsilon f\|_{L^p}  \leq C_{n,m} \int_{\R^{n}}  
		\la k\ra^{n-2m}  
		\sum_{j =0}^2   \|D_k^{j }f(k,\cdot)\|_{L^p}  \frac{dk}{|k|^{n-1}}.
		$$
		For $n=2m+1$, we have 
		$$ \| G_\epsilon f\|_{L^p} 
		\leq C_{n,m} \int_{\R^{n}}  
		\la k\ra \min(1,|k|)^{-\frac12}  
		\sum_{j =0}^3  \|D_k^{j }f(k,\cdot)\|_{L^p}   \frac{dk}{|k|^{n-1}}.
		$$
		Moreover, $ G_\epsilon f\to G_0 f$  in $L^p$ as $\epsilon \to 0^+$.
	\end{lemma}
	\begin{proof}
		Note that
		$$\big\|  F_\epsilon( t,\omega,x) \big\|_{L^p_x} \les    \int_0^\infty s^{n-2m}  \| \sup_\epsilon h_{s\omega,\epsilon}\|_{L^1}
		\|f(s\omega, \cdot)\|_{L^p} \, ds  \\ \les     \int_0^\infty s^{n-2m}  
		\|f(s\omega, \cdot)\|_{L^p} \, ds .
		$$
		For $t>1$, and $n>2m+1$, we integrate by parts twice in the $s$ integral to obtain 
		$$  |F_\epsilon( t,\omega,x) | \les  \frac{1}{t^2}\int_{\R^ n}\int_0^\infty \big|   \partial_s^2\big( s^{n-2m} h_{s\omega,\frac{\epsilon t }{2 s^{2m-1}}}(y ) f(s\omega,x-y +t\omega)\big)\big|  \, ds\, dy. 
		$$ 
		Let $H_{s\omega }(y)=|\sup_{\epsilon>0, j=0,1,2 } s^j \partial_s^j  h_{s\omega,\frac{\epsilon t }{2 s^{2m-1}}}(y )|$.  Using this we obtain the bound 
		\begin{multline*}  
		|F_\epsilon( t,\omega,x) | \les  
		\frac{1}{t^2}\int_{\R^n}\int_0^\infty \la s\ra^2s^{n-2m-2} H_{s\omega}(y) \sum_{j=0}^2 \big|   \partial_s^j f(s\omega,x-y +t\omega) \big|  \, ds\, dy \\
		\les  \frac{1}{t^2} \int_{ \R^n } \int_0^\infty  H_{s\omega}(y) \la s\ra^{n-2m}    \sum_{j=0}^2 \big|   \partial_s^j f(s\omega,x-y +t\omega) \big|  \, ds\, dy. 
		\end{multline*}
		
		By Lemma~\ref{lem:gk}, $\|H_{s\omega}\|_{L^1}\les 1$, therefore uniformly in $t$ and $\omega$, we have  
		$$\big\|  F_\epsilon( t,\omega,x) \big\|_{L^p_x} \les \frac1{\la t\ra^2}
		\int_0^\infty \la s\ra^{n-2m}   \sum_{j=0}^2 \big\|   \partial_s^j f(s\omega,\cdot)\big\|_{L^p}  \, ds,
		$$
		which implies the claim for $G_\epsilon f$ when $n>2m+1$. The convergence of $G_\epsilon f$ to $G_0f$ in $L^p$ also follows by applying the same argument with $ h_{s\omega,\frac{\epsilon t}{2s^{2m-1}}}-h_{s\omega}$ replacing $ h_{s\omega,\frac{\epsilon t}{2s^{2m-1}}} $ and using dominated convergence theorem. 
		
		We now consider the case $n=2m+1$. For $t\gg 1$, after an integration by parts, we have
		$$ F_\epsilon( t,\omega,x)   =-\frac{2i}t \int_0^\infty  e^{-is t/2} \partial_s[s  h_{s\omega,\frac{\epsilon t}{2s^{2m-1}}}* f(s\omega, \cdot)(x +t\omega ) ] \, ds. 
		$$
		We cannot integrate by parts again to gain another power of $t$ in this case. Therefore we utilize the identity (with $ K(s) =\partial_s[s  h_{s\omega,\frac{\epsilon t}{2s^{2m-1}}}* f(s\omega, \cdot)(x +t\omega ) ]$)
		$$
		\int_0^\infty  e^{-is t/2} K(s) ds= \frac12 \int_0^{2\pi/t}  e^{-is t/2} K(s) ds + \frac12  \int_{0}^\infty  e^{-i (s+2\pi/t) t  /2  } [K(s+2\pi/t) -K(s )] ds.
		$$ 
		This implies that 
		\begin{multline*}
		\Big\|\int_0^\infty  e^{-is t/2} K(s) ds\Big\|_{L^p_x}
		\les \\ \int_0^{2\pi/t}  \|K(s)\|_{L^p_x} ds+
		\int_{0}^\infty  (\|K(s+2\pi/t)\|_{L^p_x}+\|K(s)\|_{L^p_x})^\frac12 \Big(\int_{s }^{s+2\pi/t} \big\|\partial_\rho K(\rho) \big\|_{L^p_x} d\rho\Big)^\frac12 ds\\
		\les t^{-\frac12} \sup_{0<s<1} \|K(s)\|_{L^p_x}+t^{-\frac12}\int_0^\infty \big[\sup_{s<\rho<s+1} \|K(\rho)\|_{L^p}\big]^{1/2} \big[\sup_{s<\rho<s+1} \|\partial_\rho K(\rho)\|_{L^p}\big]^{1/2} ds. 
		\end{multline*}
		Note that
		$$
		\|K(\rho)\|_{L^p_x} \les \la \rho\ra   \big(\|f(\rho\omega,\cdot)\|_{L^p} +\|\partial_\rho f(\rho\omega,\cdot)\|_{L^p}\big)
		$$
		$$
		\big\|\partial_\rho K(\rho) \big\|_{L^p_x} \les \la \rho\ra  \min(1,\rho)^{-1 }    \big(\|f(\rho\omega,\cdot)\|_{L^p} +\|\partial_\rho f(\rho\omega,\cdot)\|_{L^p}+\|\partial^2_\rho f(\rho\omega,\cdot)\|_{L^p}\big).
		$$
		Therefore,
		$$
		\Big\|\int_0^\infty  e^{-is t/2} K(s) ds\Big\|_{L^p_x}\les t^{-\frac12}
		\int_0^\infty \la s\ra \min(1,s)^{-\frac12} \sup_{s<\rho<s+1} \sum_{j=0}^2 \|\partial_\rho^j f(\rho \omega,\cdot )\|_{L^p} ds.
		$$
		Noting that, for $s<\rho<s+1$ 
		$$
		\sum_{j=0}^2 \|\partial_\rho^j f(\rho \omega,\cdot )\|_{L^p} \leq \sum_{j=0}^2 \|\partial_s^j f(s \omega,\cdot )\|_{L^p} +\int_s^{s+1} \sum_{j=0}^3 \|\partial_\rho^j f(\rho \omega,\cdot )\|_{L^p}\, d\rho,
		$$
		and applying Fubini's theorem yield the claim bounding $G_\epsilon$ in $L^p$. Convergence in $L^p$ follows similarly. 
	\end{proof}

	We now return to the operator $Z_J$ defined in \eqref{eq:ZJfin}. For fixed $k_1,\ldots k_{J-1}$, the inner most integral is  $G_{\epsilon_J} \widetilde f_{k_J} $  where 
	$\widetilde f_{k_J}(k_J,x)= e^{ik_J\cdot x}  K_J(k_1, k_2,\dots, k_J) f(x)$. By Lemma~\ref{lem:Geps}, it  converges to $G_0 \widetilde f_{k_J} $ in $L^p$ for $f\in \mathcal S$.  Using Lemma~\ref{lem:Geps}, we also  take $\epsilon_{J-1},\ldots, \epsilon_1\to 0^+$ to obtain
	\be\label{ZJlimit}
	Z_Jf(x) =  \int_{ \R^{n} } T_{k_1 }^m\bigg\{  \int_{ \R^{n} } T_{k_2 }^m\bigg\{ \cdots \int_{\R^n}  
	T_{k_J}^m \widetilde f_{k_J} \, dk_J \bigg\}\cdots   \bigg\} dk_2 \bigg\} dk_1.
	\ee
	We rewrite the inner most integral using \eqref{eq:geps} (with $\epsilon=0$) as  
	\begin{multline}\label{eqn:G0 fkx}
	G_0 \widetilde f_{k_J} (x)\\
	=\int_{\R^n}\frac{i}{2|k_J|^{2m-1}}\int_0^{\infty} \int_{ \R^{ n} } e^{-i\frac{|k_J|t_J}2}  h_{k_J}(y_J ) e^{ik_J\cdot(x-y_J+t_J\omega_J)}  K_J(k_1,  \dots,  k_J) f(x-y_J +t_J\omega_J)  \, dy_J  \, dt_J\,dk_J\\
	=\int_0^{\infty} \int_{S^{n-1}}\int_{\R^n}  \int_0^\infty   \frac{is_J^{n-2m}}{2}  e^{ i\frac{s_Jt_J}2+is_J\omega_J\cdot(x-y_J)}  h_{s_J\omega_J  }(y_J )  K_J(k_1,  \dots, s_J\omega_J)  f( x-y_J +t_J\omega_J)  \, ds_J\, dt_J\, dy_J\,  d\omega_J.
	\end{multline}
	Letting $t_J+2\omega_J\cdot(x-y_J) \to -t_J$, we have
	$$ 
	\eqref{eqn:G0 fkx}=\frac{i}2 \int_{S^{n-1}}  \int_{\R^n} \int_{-\infty}^{-2\omega_J\cdot(x -y_J)}  F_J(k_1,\ldots,k_{J-1},t_J,y_J,\omega_J) f( \overline{x}-\overline{y_J} -t_J\omega_J)   \, dt_J  \, dy_J \,d\omega_J,
	$$
	where $\overline x=x-2\omega_J(x\cdot \omega_J)$ and 
	$$
	F_J(k_1,\ldots,k_{J-1},t_J,y_J,\omega_J)= \int_0^\infty  s_J^{n-2m}  e^{ i\frac{s_Jt_J}2 }  h_{s_J\omega_J  }(y_J )   K_J(k_1,  \dots,k_{J-1}, s_J\omega_J)    \, ds_J.
	$$
	Now, using \eqref{eq:geps} (with $\epsilon=0$) we rewrite the integral in $k_{J-1}$ in \eqref{ZJlimit} to obtain
	$$ 
	\eqref{eqn:G0 fkx}=\bigg(\frac{i}{2}\bigg)^2 \int_{S^{n-1}\times\R^n\times(0,\infty)}\int_{S^{n-1}\times\R^{ n}\times (-\infty,\sigma_{J-1})} F_{J-1}  f( \overline  x -\gamma_{J-1})   \, dt_J dy_J   d\omega_J   dt_{J-1}  dy_{J-1}  d\omega_{J-1},
	$$
	where for  $j=1,\dots, J-1$, 
	$$\gamma_j:=\overline{y_J}+t_J\omega_J +\sum_{\ell=j}^{J-1}\overline{y_{\ell}-t_{\ell}\omega_{\ell}},   \qquad \sigma_j=-2\omega_J \cdot(x-y_J-\sum_{\ell=j}^{J-1}(y_{\ell}-t_\ell\omega_\ell)),$$ and 
	\begin{multline*}
	F_{J-1} = F_{J-1}(k_1, \dots, k_{J-2},t_{J-1},\omega_{J-1},y_{J-1},t_J,\omega_J,y_J)\\
	:=
	\int_{(0,\infty)^2} \prod_{j=J-1}^J \big[s_j^{n-2m} e^{- i\frac{ s_jt_j}2  }  h_{s_j\omega_j} (y_j )\big]   K_J(k_1, \dots, k_{J-2},s_{J-1}\omega_{J-1},s_J\omega_J)   \,ds_J\,ds_{J-1}.
	\end{multline*}
	Continuing in this manner we have
	$$
	Z_Jf(x)=   \big(\frac{i}{2}\big)^J \int_{(S^{n-1}\times\R^n\times(0,\infty))^{J-1}}\int_{S^{n-1}\times\R^{ n}\times (-\infty,\sigma_{1})} F \,  f( \overline{x} -\gamma_{1})  dt_{J }  dy_{J }  d\omega_{J} \cdots  dt_1 dy_1 d\omega_1,  
	$$
	where 
	\begin{multline*}
	F =F (t_{ 1},\omega_{ 1},y_{ 1}, \dots,  t_J,\omega_J,y_J)\\
	:=
	\int_{(0,\infty)^J} \prod_{j= 1}^J \big[s_j^{n-2m} e^{- i\frac{ s_jt_j}2  }  h_{s_j\omega_j} (y_j )\big]   K_J(s_{ 1}\omega_{ 1}, \dots,  s_J\omega_J)   \,ds_J \cdots ds_{ 1}.
	\end{multline*}
	Taking the absolute values and then extending the integrals in $t_j$, $j=1,2,\dots, J$ to $\R$, we have 
	$$
	|Z_Jf(x)|\les   \int_{(S^{n-1}\times\R^n\times\R)^{J}} |F  (t_{ 1},\omega_{ 1},y_{ 1}, \dots,  t_J,\omega_J,y_J)|   |  f( \overline{x}-\gamma_{1})|  dt_{J }  dy_{J }  d\omega_{J} \cdots  dt_1 dy_1 d\omega_1. 
	$$
	Therefore, by Minkowski's integral inequality and noting that $x\to\overline x$ is an isometry), we have
	$$
	\|Z_Jf\|_{L^p}\les \|F  \|_{L^1((S^{n-1}\times\R^n\times\R)^{J})}\| f\|_{L^p}.
	$$
	The following lemma finishes the proof of $L^p$ boundedness of $Z_J$.
	\begin{lemma} For $2m<n<4m-1$, we have  
		$$\|F \|_{L^1((S^{n-1}\times\R^n\times\R)^{J})}\leq C^J    \| \la \cdot \ra^{ \frac{4m+1-n}{2}+} V(\cdot )\|_{L^2}^J, $$
		for $n=4m-1$, we have 
		$$\|F \|_{L^1((S^{n-1}\times\R^n\times\R)^{J})}\leq C^J  \|  \la x\ra^{1+} V\|_{H^{0+}}^{J}, $$
		for $n>4m-1$ and $\sigma>\frac{n-2m}{n-1}$, we have 
		$$\|F \|_{L^1((S^{n-1}\times\R^n\times\R)^{J})}\leq C^J  \| \mF(\la x\ra^{2\sigma} V)\|_{L^{\frac{n-1}{n-2m}-}}^J.$$
		Here $C$ depends on $n,m$ and the actual values of $\pm$ signs. 
	\end{lemma}
	\begin{proof}
		We write $F$ as a sum of $2^{J}$ operators of the form (for each subset $\mJ$ of $\{1,2,...,J\}$)
		$$
		F_{\mJ}  (t_{ 1},\omega_{ 1},y_{ 1}, \dots,  t_J,\omega_J,y_J) = F(t_{ 1},\omega_{ 1},y_{ 1}, \dots,  t_J,\omega_J,y_J) \big[\prod_{j\in\mJ}
		\chi(y_j) \big]  \big[\prod_{j\not \in\mJ}
		\widetilde\chi(y_j) \big].
		$$
		It suffices to prove that each $F_\mJ$ satisfies the claim.
		
		Fix $r\geq 2$ and $\frac1q+\frac1r=1$. By Hausdorff-Young inequality,  we have (with $L^p(\Omega)L^q(D)=L^P(\Omega,L^q(D))$)
		\begin{multline*}\|F_{\mJ}\|_{L^1 (S^{n-1}\times\R^n)^{J} L^r(\R^{J})}\les  
		\int_{(S^{n-1}\times\R^n)^{J}}  \Big[\int_{(0,\infty)^{J} } \big[\prod_{j=1}^J
		s_j^{n-2m}   h_{s_j\omega_j} (y_j )   \big]^q    \times\\ |  K_J(s_{ 1}\omega_{ 1}, \dots,  s_J\omega_J)|^q ds_1\dots ds_J\Big]^{1/q}\big[\prod_{j\in\mJ}
		\chi(y_j) \big]  \big[\prod_{j\not \in\mJ}
		\widetilde\chi(y_j) \big]d\vec y d\vec \omega.
		\end{multline*}
		Note that, by \eqref{eqn:homega bound} in the proof of Lemma~\ref{lem:gk} above (for $0<\delta\leq 1$)
		$$|h_{s\omega}(y)|\les s^n\min((s|y|)^{-n-\delta}, (s|y|)^{-n+\delta})  \les \chi(y)  |y|^{-n+\delta} s^\delta  + \widetilde \chi(y)  |y|^{-n-\delta} s^{-\delta}.
		$$
		Since $\chi(y)  |y|^{-n+\delta} \in L^1$ and $\widetilde \chi(y)  |y|^{-n-\delta}\in L^1$ for any $\delta>0$, we can bound the norm above by  
		$$\int_{(S^{n-1})^{J}}  \Big[\int_{(0,\infty)^{J}}\big[\prod_{j\in\mJ}^J s_j^{(n-2m+\delta)q}    \big] 
		\big[\prod_{j\not \in\mJ}^J s_j^{(n-2m-\delta)q}    \big]  |  K_J(s_{ 1}\omega_{ 1}, \dots,  s_J\omega_J)|^q d\vec s \Big]^{1/q} d\vec \omega.
		$$
		By Holder in $\omega_j$ integrals we conclude that
		\begin{multline} \label{eq:FL1Lr}
		\|F \|_{L^1 (S^{n-1}\times\R^n)^{J} L^r(\R^{J})}\les\Big[\int_{\R^{nJ}}    \big[\prod_{j\in\mJ}^J|k_j|^{(n-2m+\delta)q-n+1}     \big] 
		\big[\prod_{j\not \in\mJ}^J |k_j|^{(n-2m-\delta)q-n+1}    \big]  \times \\  |  K_J(k_1, \dots,  k_J )|^q dk_1\dots dk_J\Big]^{1/q}.  
		\end{multline} 
		Similarly, (here $\alpha_j=0$  or $1$ independently) 
		\begin{multline*}
		\|t_1^{\alpha_1}\dots t_J^{\alpha_J}F_\mJ\|_{L^1 (S^{n-1}\times\R^n)^{J} L^r(\R^{J})}\les
		\\\int_{(S^{n-1}\times\R^n)^{J}}  \Big[\int_{(0,\infty)^{J} }\Big|\partial_{s_1}^{\alpha_1}\dots \partial_{s_J}^{\alpha_J}\prod_{j= 1}^J \big(s_j^{n-2m}    h_{s_j\omega_j} (y_j ) \big)   \times\\ |  K_J(s_{ 1}\omega_{ 1}, \dots,  s_J\omega_J)|^q ds_1\dots ds_J\Big]^{1/q}\big[\prod_{j\in\mJ}
		\chi(y_j) \big]  \big[\prod_{j\not \in\mJ}
		\widetilde\chi(y_j) \big]d\vec y d\vec \omega.
		\end{multline*}
		Since $\partial_s h_{s\omega}$ satisfies the same bounds as $\frac1s h_{s\omega}$, proceeding as above, we obtain the  estimate  
		\begin{multline*}\|t_1^{\alpha_1}\dots t_J^{\alpha_J}F_\mJ \|_{L^1 (S^{n-1}\times\R^n)^{J} L^r(\R^{J})}\les 
		\Big[\int_{\R^{nJ}}  \big[\prod_{j\in\mJ} |k_j|^{(n-2m+\delta)q-n+1}     \big] 
		\big[\prod_{j\not \in\mJ}  |k_j|^{(n-2m-\delta)q-n+1}    \big]  \times\\ \Big|  \prod_{j=1}^J(\nabla_{k_j}^{\alpha_j}+|k_j|^{-\alpha_j })  K_J(k_1, \dots,  k_J )\Big|^q  dk_1\dots dk_J\Big]^{1/q}.  
		\end{multline*}
		Using Hardy's inequality, this implies that  
		\begin{multline}\label{eq:alphaFL1Lr}\|t_1^{\alpha_1}\dots t_J^{\alpha_J}F_\mJ  \|_{L^1 (S^{n-1}\times\R^n)^{J} L^r(\R^{J})}\les 
		\Big[\int_{\R^{nJ}}  \big[\prod_{j\in\mJ} |k_j|^{(n-2m+\delta)q-n+1}     \big] 
		\big[\prod_{j\not \in\mJ}  |k_j|^{(n-2m-\delta)q-n+1}    \big]  \times\\ \Big|  \prod_{j=1}^J \nabla_{k_j}^{\alpha_j}  K_J(k_1, \dots,  k_J )\Big|^q  dk_1\dots dk_J\Big]^{1/q}.  
		\end{multline}
		Let $2m<n<4m-1$. Applying \eqref{eq:FL1Lr} with $0< \delta\ll 1$ and $q=r=2$, we obtain 
		$$
		\|F_\mJ\|^2_{L^1 (S^{n-1}\times\R^n)^{J} L^2(\R^{J})}\les  \int_{\R^{nJ}}   \big[\prod_{j\in\mJ} |k_j|^{ n-4m+1+2\delta}     \big] 
		\big[\prod_{j\not \in\mJ}  |k_j|^{ n-4m+1-2\delta }    \big]    |  K_J(k_1, \dots,  k_J )|^2 d\vec k.  
		$$
		Note that by Hardy's inequality the integral in $k_J$ is bounded by 
		$$
		\int | |D_{k_J}|^{\frac{4m-1-n}{2}\pm\delta} \widehat V(k_{J-1}-k_J)|^2 dk_J\les   \| \la \cdot \ra^{\frac{4m-1-n}{2}\pm\delta} V(\cdot )\|_{L^2}^2\les \| \la \cdot \ra^{\frac{4m-1-n}{2}+\delta} V(\cdot )\|_{L^2}^2. 
		$$
		Repeated application of this inequality yields
		$$
		\|F_\mJ\|_{L^1 (S^{n-1}\times\R^n)^{J} L^2(\R^{J})}\les    \| \la \cdot \ra^{\frac{4m-1-n}{2}+\delta} V(\cdot )\|_{L^2}^J.
		$$
		Similarly, applying \eqref{eq:alphaFL1Lr} with $r=q=2$  and  $0<\delta\ll 1$ yield
		$$
		\|t_1^{\alpha_1}\dots t_J^{\alpha_J}F_\mJ  \|_{L^1 (S^{n-1}\times\R^n)^{J} L^2(\R^{J})} \les     \| \la \cdot \ra^{2+\frac{4m-1-n}{2}+\delta} V(\cdot )\|_{L^2}^J.
		$$
		Writing 
		$$\prod_{j=1}^J (1+|t_j|) = \sum_{\alpha_1,\dots, \alpha_J\in\{0,1\}} |t_1^{\alpha_1}\dots t_J^{\alpha_J}|,$$ these inequalities imply with that 
		$$\big\|\prod_{j=1}^J \la t_j\ra F  \big\|_{L^1 (S^{n-1}\times\R^n)^{J} L^2(\R^{J})}\les  \| \la \cdot \ra^{2+\frac{4m-1-n}{2}+\delta} V(\cdot )\|_{L^2}^J,
		$$
		which by multilinear complex interpolation leads to 
		$$\big\|\prod_{j=1}^J \la t_j\ra^{\frac12+} F_\mJ \big\|_{L^1 (S^{n-1}\times\R^n)^{J} L^2(\R^{J})}\les  \| \la \cdot \ra^{1+\frac{4m-1-n}{2}+\delta+} V(\cdot )\|_{L^2}^J.
		$$
		This proves the claim for $n<4m-1$ by Cauchy-Schwarz in $t$ integrals.
		
		For $n=4m-1$,  with $q=2-$, $r=2+$,  \eqref{eq:FL1Lr} implies
		$$
		\|F_\mJ \|^{2-}_{L^1 (S^{n-1}\times\R^n)^{J} L^{2+}(\R^{J})}\les  \int_{\R^{nJ}}     
		\big[\prod_{j\not \in\mJ}^J |k_j|^{0-}   \big]   |  K_J(k_1, \dots,  k_J )|^{2-} dk_1\dots dk_J.
		$$
		By Hardy's inequality, the integral in $k_J$ is
		\begin{multline*}  
		 \les \int \big||D_{k_J}|^{0+} \mF(V(\cdot)e^{ik_{J-1}\cdot})(k_J)\big|^{2-}  dk_J
		 \les \int \big|  \mF(\la \cdot\ra^{0+} V(\cdot)e^{ik_{J-1}\cdot})(k_J)\big|^{2-}dk_J \\ \les \int \big|  \mF(\la \cdot\ra^{0+} V(\cdot))(k_J)\big|^{2-}dk_J
\les \Big[ \int  \la k_J\ra^{0+}\big|  \mF(\la \cdot\ra^{0+}  V(\cdot))(k_J)\big|^{2}dk_J\Big]^{\frac{2-}2}		\les   \|  \la \cdot \ra^{0+} V(\cdot ) \|_{H^{0+}}^{2-}.  
		\end{multline*}   
	Repeating the same argument in the remaining variables yield 
	$$
		\|F_\mJ \|_{L^1 (S^{n-1}\times\R^n)^{J} L^{2+}(\R^{J})}\les \|  \la \cdot \ra^{0+} V(\cdot ) \|_{H^{0+}}^J.
		$$	
		Similar modifications in the other inequalities imply the claim in this case. 
		
		When $n>4m-1$, we apply the inequalities with $0<\delta \ll 1$ and $q=\frac{n-1-\delta}{n-2m} $, $r=\frac{n-1-\delta}{2m-1-\delta}$ to obtain
		$$\|F_\mJ \|_{L^1 (S^{n-1}\times\R^n)^{J} L^r(\R^{J})}\les \Big[\int_{\R^{nJ}}   \prod_{j\not\in \mJ }   |k_j|^{0- }    | K_J(k_1, \dots,  k_J )|^q dk_1\dots dk_J\Big]^{1/q}
		\les \|\mF(\la \cdot \ra^{0+} V(\cdot))\|_{L^q}^J. $$
		Similarly, we obtain 
		$$
		\|t_1^{\alpha_1}\dots t_J^{\alpha_J}F_\mJ\|_{L^1 (S^{n-1}\times\R^n)^{J} L^r(\R^{J})}\les \|\mF(\la \cdot \ra^{2+} V(\cdot))\|_{L^q}^J, 
		$$
		which implies  that 
		$$
		\big\|\prod_{j=1}^J \la t_j\ra F_\mJ \big\|_{L^1 (S^{n-1}\times\R^n)^{J} L^{\frac{n-1-\delta}{2m-1-\delta}}(\R^{J})}   \les 
		\| \mF(\la x\ra^{2+} V)\|_{L^{\frac{n-1-\delta}{n-2m}}}^J.  
		$$
		Interpolating the two bounds  we obtain (with $\sigma>\frac{n-2m}{n-1-\delta}$)
		$$
		\big\|\prod_{j=1}^J \la t_j\ra^\sigma F_\mJ\big \|_{L^1 (S^{n-1}\times\R^n)^{J} L^{\frac{n-1-\delta}{2m-1-\delta}}(\R^{J})}   \les 
		\| \mF(\la x\ra^{2\sigma} V)\|_{L^{\frac{n-1-\delta}{n-2m}}}^J, 
		$$
		which implies the claim by H\"older's inequality in $t$ integrals.  
		
	\end{proof}
	
Keeping track of the relationship between $q,r,\sigma$ and $\delta$ in the proof above leads to the statement in Theorem~\ref{thm:Born}.

\section{Low Energies: Odd dimensions}\label{sec:low}

  Throughout this section we consider odd dimensions $n$, as the Schr\"odinger resolvent has a closed form representation, \eqref{eqn:R0 explicit0}, that is entire. We prove that the low energy part of the wave operators are bounded on the range $1<p<\infty$ for odd $n$. We show in Section~\ref{sec:even} how to adapt the arguments here to account for the logarithmic singularities present in even dimensions.  Further, in Section~\ref{sec:low2} we show that for odd $n$ it is possible to capture boundedness on the endpoints of $p=1,\infty$. 

Having controlled the contribution of the Born series terms to \eqref{eqn:wave op defn}, to establish the claim of Theorem~\ref{thm:main} we need to show the boundedness of the tail of the Born series in \eqref{eqn:born identity}.  Noting that spectral localization, multiplying by the cut-off $\chi(\lambda)$ in \eqref{eqn:wave op defn} is bounded on $L^p$,  we need only control the contribution of
$$
-\frac{m}{\pi i} \int_0^\infty \chi(\lambda) \lambda^{2m-1}\mR_V^+(\lambda^{2m})V[\mR_0^+-\mR_0^-](\lambda^{2m})\, d\lambda.
$$
With $v=|V|^{\f12}$, $U(x)=1$ if $V(x)\geq 0$ and $U(x)=-1$ if $V(x)<0$, we define $M^+(\lambda)=U+v\mR_0^+(\lambda^{2m})v$.  We also define $w(x)=U(x)v(x)$.
Using the symmetric resolvent identity, one has
$$
\mR_V^+(\lambda^{2m})V=\mR_0^+(\lambda^{2m})vM^+(\lambda)^{-1}v,
$$
which is valid in a sufficiently small neighborhood of $\lambda=0$.   We show

\begin{prop}\label{prop:low energy}
	
	Let $n>2m$ be odd.  If $|V(x)|\les \la x\ra^{-\beta}$ for some $\beta>n+2$, then the operator defined by
	$$
	-\frac{m}{\pi i} \int_0^\infty \chi(\lambda) \lambda^{2m-1}\mR_0^+(\lambda^{2m})vM^+(\lambda)^{-1}v[\mR_0^+-\mR_0^-](\lambda^{2m})\, d\lambda
	$$
	extends to a bounded operator on $L^p(\R^n)$ for all $1<p<\infty$.
	
\end{prop}

We utilize the representation of the $m^{th}$ order resolvent frequently.  for notational convenience we denote $(\frac{d}{dr})^n F(r)$ by $F^{(n)}(r)$.

\begin{lemma}\label{prop:F} Let $n>2m$ be odd. Then,   we have the following representation of the free resolvent
	$$
	\mR_0^+(\lambda^{2m})(y,u)
	=  \frac{e^{i\lambda|y-u|}}{|y-u|^{n-2m}} F(\lambda |y-u| ).$$
	Here   $|F^{(N)}(r)|\les  \la r\ra^{\frac{n+1}2 -2m-N}$, $N=0,1,2,...$ 
\end{lemma}

\begin{proof}
	
	By the splitting identity, and \eqref{eqn:R0 explicit0} we have
	\begin{align*}
	\mR_0^+(\lambda^{2m})(y,u)&=  \frac{1}{ m\lambda^{2m-2}}\big[R_0^+ ( \lambda^2)(y,u)+
	\sum_{\ell=1}^{m-1} \omega_\ell R_0^+ ( \omega_\ell \lambda^2)(y,u)\big]\\
	&=  \frac{e^{i\lambda| y-u|}}{ m\lambda^{2m-2}|y-u|^{n-2}}\big[P_{\frac{n-3}2} (\lambda|y-u|)   +
	\sum_{\ell=1}^{m-1} \omega_\ell e^{i(\omega_\ell^{\f12}-1)\lambda|y-u|} P_{\frac{n-3}2} (\omega_\ell^{\f12} \lambda |y-u| )   \big] 
	\end{align*}
	Here $P_{k}(s)$ indicates a polynomial of degree $k$ in $s$, the exact coefficients are not important.  Therefore,
	$$
	F(r)  =  \frac{1}{ mr^{2m-2} }\big[P_{\frac{n-3}2} (r)   +
	\sum_{\ell=1}^{m-1} \omega_\ell e^{i(\omega_\ell^{\f12}-1)r} P_{\frac{n-3}2} (\omega_\ell^{\f12}r)   \big]=:\frac{g(r)}{r^{2m-2}}.
	$$
	Note that $g$ is entire and bounded by a constant multiple of $\la r\ra^{\frac{n-3}2}$ on the positive real line. Moreover,  $|\partial_r^N g(r)|\les \la r\ra^{\frac{n-3}2-N}$ for each $N\in\mathbb N$ and $r>0$. By a Taylor series expansion, see for example Proposition~2.4 in \cite{soffernew}, the resolvent is bounded in $\xi$ as $|\xi|\to 0$ between suitable weighted $L^2$ spaces, and has a series expansion in $|\xi|\,|y-u|$ near $\xi=0$.   This implies that $g$ has a zero of degree $\geq 2m-2$ at $0$, which implies the first claim.
\end{proof}

To prove Proposition~\ref{prop:low energy}, we need to understand the operator $M^+(\lambda)^{-1}$.  By the assumption that zero energy is regular,  $M^{+}(\lambda)^{-1}$ is a bounded operator.  To show this, we use the following low energy bounds on the resolvent.
\begin{lemma}\label{lem:low resolv sup}
	Let $n>2m $ be odd. 
	We have the following bounds on the derivatives of the resolvent.  For $k=1,2,\dots,$ we have
	$$
	\sup_{0<\lambda <1} |\lambda^{k-1}\partial_\lambda^k \mR_0(\lambda^{2m})(x,y)| \les 
	|x-y|^{2m+1-n}+|x-y|^{k-(\frac{n-1}{2})}.
	$$
	
\end{lemma} 

\begin{proof}
	
	In all cases we use the expansions in Lemma~\ref{prop:F}.  By the product and chain rules, we have
	\begin{multline*}
	\big|\partial_\lambda^k \mR_0(\lambda^{2m})(x,y)\big|=
	\bigg|\frac{1}{|x-y|^{n-2m}} \sum_{\ell=0}^{k} \binom{k}{\ell} \big(\partial_\lambda^{k-\ell}e^{i\lambda |x-y|}\big) \big(\partial_\lambda^\ell F(\lambda |x-y|) \big)\bigg|\\
	\les |x-y|^{2m-n+k} \sum_{\ell=0}^k \la \lambda |x-y| \ra^{\frac{n+1}{2}-2m-\ell} \les |x-y|^{2m-n+k} \la \lambda |x-y|\ra^{\frac{n+1}{2}-2m}.
	\end{multline*}
	From here, it follows that  
	\begin{align*}
	\big|\lambda^{k-1}\partial_\lambda^k \mR_0(\lambda^{2m})(x,y)\big| \les \lambda^{k-1}|x-y|^{k+2m-n}\la \lambda |x-y|\ra^{\frac{n+1}{2}-2m}.
	\end{align*}
	When $\lambda |x-y|\leq 1$, we cannot use the terms in the bracket, but instead rearrange to see
	\begin{align*}
	\chi(\lambda|x-y|)\big|\lambda^{k-1}\partial_\lambda^k \mR_0(\lambda^{2m})(x,y)\big| \les \chi(\lambda|x-y|) (\lambda|x-y|)^{k-1}|x-y|^{2m+1-n} \les |x-y|^{2m+1-n}  .
	\end{align*}
	Here we used that $k-1\geq 0$.
	When $\lambda |x-y|\geq 1$, we have
	\begin{align*}
	\widetilde\chi(\lambda|x-y|)\big|\lambda^{k-1}\partial_\lambda^k \mR_0(\lambda^{2m})(x,y)\big| \les \widetilde \chi(\lambda|x-y|) (\lambda|x-y|)^{k-1-2m+\frac{n+1}{2}}|x-y|^{2m+1-n}  .
	\end{align*}
	Here we consider cases, either $k-1-2m+\frac{n+1}{2}<0$ hence the first term is bounded by one and we have the bound $|x-y|^{1+2m-n}$. On the other hand, if $k-1-2m+\frac{n+1}{2}\geq 0$ we bound by
	$$
	\widetilde \chi(\lambda|x-y|) \lambda^{k-1-2m+\frac{n+1}{2}}|x-y|^{k-(\frac{n-1}{2})}  .
	$$
	Since the exponent on $\lambda$ is non-negative, taking the supremum on $0<\lambda<1$ yields the bound of $|x-y|^{k-(\frac{n-1}{2})}$. 
\end{proof}

To control the low energy, we define the following terms.  First, we define an operator $T:L^2\to L^2$ with integral kernel $T(\cdot, \cdot)$ to be absolutely bounded if the operator with kernel $|T(\cdot,\cdot)|$ is also bounded on $L^2$.  Further, we define the operator
$$
T_0:=U+v\mR_0^+(0)v=M^+(0).
$$
Here $v=|V|^{\frac12}$ and $V=vw$, recall that $|w|=v$.
By the assumption that zero energy is regular, $T_0$ is invertible, see e.g. \cite{soffernew}.

The bounds in Lemma~\ref{lem:low resolv sup} imply that the operator $R_k$ with kernel
\be\label{R_k}
R_k(x,y):=v(x)v(y)\sup_{0<\lambda<1}|\lambda^{k-1} \partial_\lambda^k \mR_0(\lambda^{2m})(x,y)|
\ee
is bounded on $L^2(\R^2)$ for $1\leq k\leq \frac{n+1}2$ provided that $|V(x)|\les \la x \ra^{-\beta}$ for some $\beta>n+2$.  We note that when $n$ is large compared to $m$, we identify $|x-y|^{2m+1-n}$ as a multiple of the fractional integral operator $I_{2m+1}:L^{2,\sigma}\to L^{2,-\sigma}$, see Propositions~3.2 and 3.3 in \cite{GV} for example.  Using the decay of $v(x)v(y)$ suffices when identifying $\sigma=\sigma'=\frac{\beta}{2}$, to apply the Propositions in \cite{GV} and establish boundedness on $L^2$.

Note that by a Neumann series expansion and the invertibility of $T_0$ we have 
$$
[M^+(\lambda)]^{-1} =\sum_{k=0}^\infty (-1)^k T_0^{-1} (E(\lambda)T_0^{-1})^{k},
$$
where $E(\lambda)=v[\mR^+_0(\lambda^{2m})-\mR^+_0(0)]v$
for $0<\lambda <\lambda_0$. By \eqref{R_k} and the mean value theorem we have 
$$
E_0(x,y):=\sup_{0<\lambda<\lambda_0}|E(\lambda)(x,y)|
\les \lambda_0 R_1(x,y)
$$
is a bounded operator on $L^2$ with norm $\les \lambda_0$.  Therefore, 
$$
\Gamma_0(x,y):=\sup_{0<\lambda<\lambda_0}| [M^+(\lambda)]^{-1}(x,y)|
$$ 
is bounded on $L^2$ for sufficiently small $\lambda_0$.

Similarly, note that by the resolvent identity the operator
$\lambda^N \partial_\lambda^N[M^+(\lambda)]^{-1} $ is a linear combination of operators of the form 
$$
[M^+(\lambda)]^{-1} \prod_{j=1}^J\big[v\big(\lambda^{k_j}\partial_\lambda^{k_j}\mR_0^{+}(\lambda^{2m})\big)v[M^+(\lambda)]^{-1}\big],
$$
where $\sum k_j=N$ and each $k_j\geq 1$. Therefore using \eqref{R_k} we see that 
\be\label{eq:MGamma}
\Gamma_N(x,y):=\sup_{0<\lambda<\lambda_0}\lambda^N | \partial_\lambda^N[M^+(\lambda)]^{-1}(x,y)|
\ee
is bounded in $L^2$ for $N=0,1,\ldots,\frac{n+1}2$ provided that $\beta >n+2$.  Further, for $N\geq 1$ we may replace $\lambda^{N}$ with $\lambda^{N-1}$, and the operator remains bounded on $L^2$.   This bound suffices to prove Proposition~\ref{prop:low energy} for $n<4m$, odd. However, for odd $n>4m$ we need to modify the approach to account for  the fact that $|x-\cdot|^{2m-n}$ is no longer locally $L^2(\R^n)$.  We iterate the Born series further and utilize the following
\begin{align}\label{eq:Alambda}
A(\lambda, z_1,z_2) =  \big[ \big(\mR_0^+(\lambda^{2m})  V\big)^{\kappa}\mR_0^+(\lambda^{2m})\big](z_1,z_2).
\end{align}  
By repeated iterations of Lemma~\ref{lem:IMRN lems} using the representation of Lemma~\ref{prop:F}, each iteration of the resolvent  smooths out $2m$ power  of the singularity.  
Selecting $\kappa$ large enough ensures that $A$ is bounded.  That is, we have
\begin{lemma}\label{lem:A lems} Fix odd $n>4m$.  If   $\kappa\in \mathbb N$ is sufficiently large depending on $n,m$ and $|V(x)|\les \la x \ra^{-\frac{n+3}2-}$, then 
	\begin{align*}
	\sup_{0<\lambda <1}|\partial_\lambda^\ell 	A(\lambda, z_1,z_2)|&\les \la z_1 \ra  \la z_2\ra,
	\end{align*}
	for $0\leq \ell\leq \frac{n+1}{2}$.   
\end{lemma}
We will prove this lemma at the end of this section.  We say an operator $K$ is admissible if its integral kernel $K(x,y)$ satisfies
$$
	\sup_{x\in \R^n} \int_{\R^n}|K(x,y)|\, dy+	\sup_{y\in \R^n} \int_{\R^n}|K(x,y)|\, dx<\infty.
$$ 
By the Schur test, an admissible operator is bounded on $L^p$ for all $1\leq p\leq \infty$.

By  iterating the Born series sufficiently many times it suffices to prove that the operator with kernel 
\begin{align*}
\int_0^\infty \lambda^{2m-1}\chi(\lambda)[\mR_0^+(\lambda^{2m})V A(\lambda)v M^{-1}(\lambda)vA(\lambda) V \mR_0^\mp(\lambda^{2m})](x,y)\, d\lambda.
\end{align*}
is bounded on $L^p$, $1<p<\infty$. Letting (recall that $|w|=v$)
$$\Gamma=w A(\lambda)v M^{-1}(\lambda)vA(\lambda) w,$$
and using Lemma~\ref{lem:A lems} and \eqref{eq:MGamma}  we see that $\Gamma$ 
  satisfies 
\be\label{eq:tildegamma}
\widetilde \Gamma (x,y):= \sup_{0<\lambda <\lambda_0} \sup_{0\leq k\leq  \frac{n+1}2} \big|\lambda^k \partial_\lambda^k \Gamma(\lambda)(x,y)\big| \les \la x\ra^{-\frac{n}2-}\la y\ra^{-\frac{n}2-},
\ee
provided that $\beta>n+2$.  Hence, Proposition~\ref{prop:low energy} is a consequence of the following bound.

\begin{lemma}\label{lem:low tail low d} Fix $n$ odd and let $\Gamma$ be a $\lambda $ dependent absolutely bounded operator. Let 
$$
\widetilde \Gamma (x,y):= \sup_{0<\lambda <\lambda_0} \sup_{0\leq k\leq  \frac{n+1}2} \big|\lambda^k \partial_\lambda^k \Gamma(\lambda)(x,y)\big|.
$$
For $2m<n<4m$ assume that $\widetilde \Gamma $ is bounded on $L^2$, and for $n>4m$  assume that $\widetilde\Gamma $ satisfies   \eqref{eq:tildegamma}. Then the operator with kernel 
	$$
	K(x,y)=\int_0^\infty \chi(\lambda) \lambda^{2m-1} [\mR_0^+(\lambda^{2m}) v \Gamma v \mR_0^-(\lambda^{2m})](x,y)   d\lambda 
	$$
	is bounded on $L^p$ for $1<p<\infty$ provided that $\beta>n$.
\end{lemma}
\begin{proof} Using the representation in Lemma~\ref{prop:F} with $r_1=|x-z_1|$ and $r_2:=|z_2-y|$ we have  
	\be \label{Kdefini}
	K(x,y)=\int_{\R^{2n}}  \frac{ v(z_1) v(z_2) }{r_1^{n-2m} r_2^{n-2m} } \int_0^\infty e^{i\lambda (r_1 -r_2 )} \chi(\lambda) \lambda^{2m-1} \Gamma(\lambda)(z_1,z_2) F(\lambda r_1)F(\lambda r_2)  d\lambda dz_1 dz_2 .
	\ee
	Using the bounds in Lemma~\ref{prop:F}, \eqref{eq:tildegamma}, and the assumption $\beta>n$, we bound the $\lambda$-integral above by 
	\begin{align}\label{kbound0}
	 \widetilde\Gamma (z_1,z_2)    \int_0^1       \frac{\lambda^{2m-1}   }{ \la \lambda r_1\ra^{2m-\frac{n+1}2} \la \lambda r_2\ra^{2m-\frac{n+1}2} } d\lambda .
	\end{align}
	Also note that by integrating by parts $N\leq \frac{n+1}2$ times   in $\lambda$ when $\lambda|r_1-r_2|>1$ and using \eqref{kbound0} when
	$\lambda|r_1-r_2|<1$, and recalling   the bounds for the derivatives of $F$, we obtain 
	\begin{align} 
	\label{kbound02}
	|K(x,y)|\les \int_{\R^{2n}} \frac{v(z_1)  \widetilde\Gamma (z_1,z_2) v(z_2)}{  r_1^{n-2m} r_2^{n-2m}} \int_0^1 \frac{ \lambda^{2m-1}   } { \la \lambda (r_1-r_2)\ra^N     \la \lambda r_1\ra^{2m-\frac{n+1}2} \la \lambda r_2\ra^{2m-\frac{n+1}2} } d\lambda  dz_1 dz_2. 
	\end{align}
Note that there are no boundary terms here since we include  the cutoff $\widetilde \chi	(\lambda (r_1-r_2))$ in the integration by parts argument above. Also note that we can choose $N$ depending on $z_1, z_2$.  We write
	$$
	K(x,y) =: \sum_{j=1}^4 K_{j}(x,y),
	$$
	where the integrand in  
	$K_1$ is  restricted  to the set $r_1,r_2\les 1$, in $K_2$ to the set $r_1 \approx r_2\gg 1 $,
	in $K_3$ to the set $r_2\gg \la r_1\ra $,  in $K_4$ to the set $r_1 \gg \la r_2\ra$.

	Note that $K_1$ is admissible using \eqref{kbound02} with $N=0$: For $n<4m$ we have 
	\begin{multline*}
		\int|K_1(x,y)| dx\les \int_{\R^{2n}} \int_{ r_1 <1} \frac{ v(z_1)  \widetilde\Gamma  (z_1,z_2)     v(z_2) }{r_1^{n-2m} r_2^{n-2m}    }   dx dz_1 dz_2\\  
		\les \|v(\cdot)|y-\cdot|^{2m-n}\|_{L^2} \|\widetilde\Gamma \|_{L^2\to L^2} \||x-\cdot|^{n-2m}v\|_{L^2}\les 1,
	\end{multline*}
	provided that $\beta>n $.  For $n>4m$, we instead have 
	\begin{multline*}
	\int|K_1(x,y)| dx\les \int_{\R^{2n}} \int_{ r_1 <1} \frac{   \la z_1\ra^{-n- }    \la z_2\ra^{-n- }}{r_1^{n-2m} r_2^{n-2m}    }   dx dz_1 dz_2  \\
	\les \big\| |\cdot|^{2m-n}\big\|_{L^1(B(0,1))} \big\|\la \cdot\ra^{-n-} |y-\cdot|^{2m-n}\big\|_{L^1}   \big\|\la \cdot\ra^{-n-}\big\|_{L^1} \les 1,
	\end{multline*}
	uniformly in $y$. The $y$-integrals can be estimated similarly.

	Similarly $K_2$ is admissible using \eqref{kbound02} with $N=2$: For $n<4m$ we have 
	\begin{multline*}
	\int|K_2(x,y)| dx\les  \int_{\R^{2n}} \int_{ r_1 \approx r_2  \gg 1 } \frac{ v(z_1)  \widetilde\Gamma (z_1,z_2)     v(z_2) }{r_1^{2n-4m}    }   \int_0^1\frac{\lambda^{2m-1}   }{ \la \lambda (r_1-r_2)\ra^2 \la \lambda r_1\ra^{4m- n-1}  }  d\lambda dx  dz_1 dz_2 
	\\ \les  \int_{\R^{2n}} v(z_1)   \widetilde\Gamma(z_1,z_2)     v(z_2)   \int_{ r_1 \approx r_2 \gg1  }  \int_0^1\frac{\lambda^{2m-1}  r_1^{4m-n-1} }{ \la \lambda (r_1-r_2)\ra^2  \la \lambda r_1\ra^{4m- n-1}  }  d\lambda d r_1  dz_1 dz_2 \\
	=\int_{\R^{2n}} v(z_1)   \widetilde\Gamma(z_1,z_2)  v(z_2)    \int_0^1\int_{ \eta \approx \lambda r_2  \gg \lambda } \frac{\lambda^{n-2m-1}  \eta^{4m-n-1} }{ \la \eta - \lambda  r_2 \ra^2  \la \eta\ra^{4m- n-1}  }  d \eta  d\lambda dz_1 dz_2  \les 1,  
	\end{multline*}
	provided that $\beta>n $. When $n>4m$ we bound the last integral by 
\begin{multline*}
 \int_{\R^{2n}}  \la z_1\ra^{-n- }    \la z_2\ra^{-n- }   \int_0^1\int_{ \eta \approx \lambda r_2  \gg \lambda } \frac{\lambda^{n-2m-1}  \eta^{4m-n-1} }{ \la \eta - \lambda  r_2 \ra^2  \la \eta\ra^{4m- n-1}  }  d \eta  d\lambda dz_1 dz_2\\
 \les \int_{\R^{2n}}  \la z_1\ra^{-n- }    \la z_2\ra^{-n-}   \int_0^1 \Big[\int_1^\infty  \frac{\lambda^{n-2m-1}  d \eta }{ \la \eta - \lambda  r_2 \ra^2   }  +  \int_{1> \eta \approx \lambda r_2  \gg \lambda } \frac{\lambda^{ 2m-2 }  d \eta }{ \la \eta - \lambda  r_2 \ra^2    }    \Big] d\lambda dz_1 dz_2 \les 1.
	\end{multline*}
The $y$-integrals can be estimated similarly.

	We will prove that $K_3$ and $K_4$ are bounded in $L^p$ for $1<p<\infty$. By symmetry we will only consider $K_3$. By using \eqref{kbound02} with $N=\frac{n+1}2$ we have the bound
	\be\label{eq:k3cases}	
 |K_3(x,y)|\les \int_{\R^{2n}} \frac{ v(z_1) \widetilde\Gamma(z_1,z_2)  v(z_2)  }{  r_1^{n-2m} r_2^{n-2m}} \int_0^1 \frac{ \lambda^{2m-1}  \la \lambda r_1\ra^{\frac{n+1}2 -2m}}  {      \la \lambda r_2\ra^{2m} } d\lambda  dz_1 dz_2. 
\ee
When $n<4m$, we bound this by
	\begin{multline*}
	|K_3(x,y)|\les \int_{\R^{2n}} \frac{  v(z_1) \widetilde\Gamma(z_1,z_2)  v(z_2) }{r_1^{n-2m}  r_2^{n-2m}      } 
	\int_0^1 \frac{ \lambda^{2m-1}}{ \la \lambda r_2\ra^{2m }  } d\lambda dz_1 dz_2
	\\ \les \int_{\R^{2n}} \frac{ v(z_1)  \widetilde\Gamma(z_1,z_2)  v(z_2) \log(  r_2 )}{ r_1^{n-2m} r_2^{n}      } 
	dz_1 dz_2
	\les  \int_{\R^{2n}} \frac{ v(z_1) \widetilde\Gamma(z_1,z_2)  v(z_2) }{ r_1^{n-2m} \la r_1\ra^{\frac{n}p-}\la r_2\ra^{\frac{n}{p^\prime}+}      } 
	dz_1 dz_2.
	\end{multline*}
	By H\"older we have 
	$$
	\Big\|\int K_3(x,y)  f(y)  dy \Big\|_{L^p} \les \|f\|_{L^p}  \Big\| \int_{\R^{2n}} \frac{  v(z_1) \widetilde\Gamma(z_1,z_2)   v(z_2) }{ |x-z_1|^{n-2m} \la x-z_1\ra^{\frac{n}p-}    } 
	dz_1 dz_2\Big\|_{L^p}.
	$$
	When $|x-z_1|>1$  the bound is easy by Minkowski integral inequality. Similarly, when  $|x-z_1|<1$ and $p<\frac{n}{n-2m}$. When $|x-z_1|<1$ and $p\geq \frac{n}{n-2m}$, we estimate the integral by 
	\begin{multline*}
	\la x\ra^{-\beta/2} \int_{\R^{2n}} \frac{  \widetilde\Gamma(z_1,z_2)  \chi_{|x-z_1|<1}  v(z_2) }{ |x-z_1|^{n-2m}  } dz_1 dz_2 \\ \les  \| \widetilde\Gamma\|_{L^2\to L^2} \|v\|_{L^2} \|  |z_1|^{2m-n}\|_{L^2_{B(0,1)}}    \la x\ra^{-\beta/2} \les \la x\ra^{-\beta/2}  \in L^p, 
	\end{multline*}
	provided that $\beta/2>\frac{n}p$, which holds if   $\beta>n$.

	For $n>4m$, we bound the $\lambda$-integral in \eqref{eq:k3cases} by   
$$   \int_0^1 \frac{ \lambda^{2m-1}  }  {      \la \lambda r_2\ra^{2m} } d\lambda +   \int_0^1 \frac{ \lambda^{\frac{n-1}2-2m}  r_1^{\frac{n+1}2 -2m}  }  {       r_2^{2m} } d\lambda  \les \frac{\log(r_2)}{r_2^{2m}}+\frac{  r_1^{\frac{n+1}2 -2m}  }  {       r_2^{2m} }.
$$
Therefore,
$$|K_3(x,y)|\les  \int_{\R^{2n}} \la z_1\ra^{-n-  }    \la z_2\ra^{-n- } \Big[\frac{ \log(r_2) }{  r_1^{n-2m} r_2^{n}} + \frac{1}{  r_1^{\frac{n-1}2} r_2^{n}}\Big]dz_1dz_2.
$$
This can be bounded as above considering the cases $|x-z_1|>1$ and $|x-z_1|<1$ separately.
\end{proof}

We now complete the proof of Proposition~\ref{prop:low energy} by proving Lemma~\ref{lem:A lems}:   
\begin{proof}[Proof of Lemma~\ref{lem:A lems}]
 Using Lemma~\ref{prop:F}, we note that when $\lambda <1$ we have (with $u_0=z_1$ and $u_{\kappa+1}=z_2$)
	\begin{multline}
	\bigg|\partial_{\lambda}^\ell \bigg(\prod_{j=1}^\kappa  \mR_0(\lambda^{2m})(u_{j-1},u_j)V(u_j) \mR_0(\lambda^{2m})(u_{\kappa},z_2)\bigg) \bigg|\\
	=\bigg|\partial_{\lambda}^\ell \bigg(e^{i\lambda \sum_{j=1}^{\kappa+1} |u_{j-1}-u_{j}|}\prod_{j=1}^\kappa \frac{F(\lambda|u_{j-1}-u_j|)  V(u_j)}{|u_{j-1}-u_j|^{n-2m}} \frac{F(\lambda|u_{\kappa}-u_{\kappa+1}|)}{|u_{\kappa}-u_{\kappa+1}|^{n-2m}} \bigg)\bigg|\\
	\les (\sum_{j=1}^{\kappa+1}|u_{j-1}-u_j|^\ell) \prod_{j=1}^\kappa \frac{\la u_{j-1}-u_j\ra^{\frac{n+1}{2}-2m}  |V(u_j)|}{|u_{j-1}-u_j|^{n-2m}} \frac{\la u_{\kappa}-u_{\kappa+1}\ra^{\frac{n+1}{2}-2m}}{|u_{\kappa}-u_{\kappa+1}|^{n-2m}}.\label{eqn:iterated l derivs}
	\end{multline}
We only consider the case when $j=\kappa+1$ in the first sum above; the other cases boils down to this case. We need to bound
$$
\int \prod_{j=1}^\kappa \frac{\la u_{j-1}-u_j\ra^{\frac{n+1}{2}-2m}  |V(u_j)|}{|u_{j-1}-u_j|^{n-2m}} \frac{\la u_{\kappa}-u_{\kappa+1}\ra^{\frac{n+1}{2}-2m}}{|u_{\kappa}-u_{\kappa+1}|^{n-2m-\ell}} du_1\dots du_{\kappa}.
$$
Note that for $ a=1,...,\lfloor n/2m\rfloor -1$, we have  
$$
\int  \frac{\la u_{0}-u \ra^{\frac{n+1}{2}-2ma}  }{|u_{0}-u|^{n-2ma}}  \la u\ra^{-\frac{n+1}2-} \frac{\la u -u_1 \ra^{\frac{n+1}{2}-2m  } }{|u -u_1|^{n-2m}} du \les \frac{\la u_{0}-u_1 \ra^{\frac{n+1}{2}-2m(a+1)}  }{|u_{0}-u_1|^{n-2m(a+1)}},  
$$
namely the power of the singularity decreases by $2m$ but the decay rate does not change. 
To see this inequality consider the cases $|u_{0}-u|<1$, $|u_{0}-u|>1$ separately and same for $|u-u_1|$. Also note that if $a\geq \lfloor n/2m\rfloor$ , then the bound is $ \la u_{0}-u_1 \ra^{-\frac{n-1}2}$.

Using this bound in $u_1,\dots, u_{\kappa-1}$ integrals, and assuming $\kappa$ is large, we obtain the bound 
$$
\int  \la u_{0}-u_\kappa \ra^{-\frac{n-1}{2}   }  \la u_\kappa\ra^{-\frac{n+3}2-} \frac{\la u_{\kappa}-u_{\kappa+1}\ra^{\frac{n+1}{2}-2m}}{|u_{\kappa}-u_{\kappa+1}|^{n-2m-\ell}}   du_{\kappa}.
$$
This is $\les 1$ if $\ell=0,1,\dots,\frac{n-1}2$. If $\ell=\frac{n+1}2$, then the bound is $\la u_{\kappa+1}\ra$.

If $j\neq 1,\kappa+1$, we start integrating from the farther end to the $j$th term and obtain the bound $\les 1$. 
\end{proof}

\section{Low Energy: Endpoint estimates in odd dimensions}\label{sec:low2}

In this section we prove that the low energy portion of the wave operators in odd dimensions is bounded at the endpoint values of $p=1,\infty$.  The proof relies on the explicit closed form representation of the odd dimensional resolvents.  Namely, we show
\begin{prop}\label{prop:improved low energy}
	
	Let $n>2m$ be odd.  If $|V(x)|\les \la x\ra^{-\beta}$ for some $\beta>n+4$, then the operator defined by
	$$
	-\frac{m}{\pi i} \int_0^\infty \chi(\lambda) \lambda^{2m-1}\mR_0^+(\lambda^{2m})vM^+(\lambda)^{-1}v[\mR_0^+-\mR_0^-](\lambda^{2m})\, d\lambda
	$$
	extends to a bounded operator on $L^p(\R^n)$ for all $1\leq p\leq \infty$.
\end{prop} 
Unlike Proposition~\ref{prop:low energy}, this proposition relies on a detailed analysis of the cancellation in $\mR_0^+-\mR_0^-$.  \\

\noindent {\bf Remark:} (Correction 08/11/2022)\\
	{\it The statement of Lemma~\ref{lem:R0cancel} should be:\\
	Let $n>2m$ be odd. We have 
	$$
	[\mR_0^+(\lambda^{2m})-\mR_0^-(\lambda^{2m})](y,u)= \lambda^{n-2m}  [e^{i\lambda |y-u|} \widetilde F_+(\lambda |y-u|)+e^{-i\lambda |y-u|} \widetilde F_-(\lambda |y-u|)],
	$$
	where $\widetilde F_\pm$ are functions satisfying 
	$$
	|\partial_r^j\widetilde F_\pm(r)|\les \la r\ra^{\frac{1-n}2-j},\,\,\,r\in \R.
	$$
This representation follows from the proof below.  It results in a slight change to the proof of Lemma~\ref{lem:low tail low d k3k4}.  More explicitly in the analysis of the operators $K_{4,j}$ which make up $K_4$ in \eqref{K4newdef}, the phase is $\lambda(\omega_{\ell}^{1/2} r_1\pm r_2)$ instead of   $\lambda \omega_\ell^{1/2}  r_1$.  The rest of the argument follows through since in this regime, $|\omega_{\ell}^{1/2}r_1\pm r_2|\approx r_1$.}\\

We start with the following 
\begin{lemma}\label{lem:R0cancel}
Let $n>2m$ be odd. We have 
$$
[\mR_0^+(\lambda^{2m})-\mR_0^-(\lambda^{2m})](y,u)= \lambda^{n-2m}   \widetilde F(\lambda |y-u|),
$$
where $\widetilde F$ is an entire function satisfying 
$$
|\partial_r^j\widetilde F(r)|\les \la r\ra^{\frac{1-n}2-j},\,\,\,r\in \R.
$$

\end{lemma} 

\begin{proof}
	
	By the splitting identity \eqref{eqn:Resolv} and the explicit form of the odd dimensional Schr\"odinger resolvent, we may write:
	\begin{multline}\label{eqn:resolv dif F}
	[\mR_0^+-\mR_0^-](\lambda^{2m})(y,u)=\frac{1}{m\lambda^{2m-2}}[R_0^+-R_0^-](\lambda^2)(y,u)\\
	=\lambda^{n-2m} \frac{e^{i\lambda|y-u|}P_{\frac{n-3}{2}}(\lambda |y-u|)-e^{-i\lambda|y-u|}P_{\frac{n-3}{2}}(-\lambda |y-u|)}{(\lambda |y-u|)^{n-2}}.
	\end{multline}
	Here $P_{\frac{n-3}{2}}(r)$ is a polynomial of order $\frac{n-3}{2}$ whose coefficients may be computed exactly.  We identify
	$$
	\widetilde F(r)=\frac{e^{ir}P_{\frac{n-3}{2}}(r)-e^{-ir}P_{\frac{n-3}{2}}(-r)}{r^{n-2}}.
	$$
	For $r>1$ the bounds  are clear.  For $0<r<1$, a careful Taylor series expansion as in \cite{Jen,GG1} shows that for $c_j\in \R$,
	\begin{multline*}
	\frac{R_0^\pm (\lambda^{2})(y,u)}{(\lambda|y-u|)^{n-2}}=c_0+c_1(\lambda|y-u|)+c_2(\lambda|y-u|)^{2}+\dots +c_{n-3}(\lambda|y-u|)^{n-3}\\
	+\sum_{j=\frac{n-2}{2}}^\infty ( c_{2j}(\pm i\lambda|y-u|)^{2j} +c_{2j+1}(\lambda|y-u|)^{2j+1} ).
	\end{multline*}
	From which we deduce, for $0<r<1$,
	$$
	\widetilde F(r)=\sum_{j=0}^{\infty} 2ic_{2j+n-2}r^{2j},
	$$
	which suffices to prove the claim.
	
\end{proof}

As in the previous section, the proposition follows from the following

\begin{lemma}\label{lem:low tail low d k3k4} Fix $n$ odd and let $\Gamma$ be a $\lambda $ dependent absolutely bounded operator. Let 
$$
\widetilde \Gamma (x,y):= \sup_{0<\lambda <\lambda_0} \Big[\big| \Gamma(\lambda)(x,y) \big| +\big|\partial_\lambda  \Gamma(\lambda)(x,y) \big| + \sup_{2\leq k\leq  \frac{n+3}2} \big|\lambda^{k-2} \partial_\lambda^k \Gamma(\lambda)(x,y)\big|\Big].
$$
For $2m<n<4m$ assume that $\widetilde \Gamma $ is bounded on $L^2$, and for $n>4m$  assume that $\widetilde\Gamma $ satisfies   \eqref{eq:tildegamma}. Then the operator with kernel 
	$$
	K(x,y)=\int_0^\infty \chi(\lambda) \lambda^{2m-1} \big[\mR_0^+(\lambda^{2m}) v \Gamma v [\mR_0^+ -\mR_0^- ] (\lambda^{2m})\big](x,y)   d\lambda 
	$$
	is admissible, and hence it is bounded  on $L^p$ for $1\leq p\leq \infty$ provided that $\beta>n$.
\end{lemma}
Note that the assumption on $\Gamma$ is stronger than the one in Lemma~\ref{lem:low tail low d}. By a straightforward modification of  Lemma~\ref{lem:low resolv sup} and Lemma~\ref{lem:A lems}, which requires $\beta>n+4$, the operator $\Gamma = vM^+(\lambda)^{-1}v$ satisfies the assumption for $2m<n<4m$.  When $n>4m$ the operator  $\Gamma=w A(\lambda)v M^{-1}(\lambda)vA(\lambda) w,$ satisfies the hypotheses for sufficiently large $\kappa$  when $n>4m$.

\begin{proof}[Proof of Lemma~\ref{lem:low tail low d k3k4}]
We define $K_1,...,K_4$ as in the proof of Lemma~\ref{lem:low tail low d} and use the notation $r_1=|x-z_1|$, $r_2=|y-z_2|$. Since we already proved the admissibility of $K_1$ and $K_2$, it remains to consider $K_3$ restricted to the region $r_2\gg\la r_1\ra$ and $K_4$ restricted to the region $r_1 \gg \la r_2\ra$. We first consider $K_4$. Using the splitting identity for the resolvent on the left and Lemma~\ref{lem:R0cancel} on the right, we write the kernel of $K_4$ as follows (ignoring constants):
\begin{multline} \label{K4newdef}
	K_4(x,y)=\sum_{\ell=0}^{m-1}\omega_\ell  \times \\
	 \int_{\R^{2n}}  \frac{v(z_1) v(z_2)\chi_{r_1 \gg \la r_2\ra}}{r_1^{n-2}}   \int_0^\infty \frac{e^{i\omega_\ell^{1/2} \lambda r_1 } P_{\frac{n-3}2}(\omega_\ell^{1/2} \lambda r_1)}{\lambda^{2m-2}} \chi(\lambda) \lambda^{2m-1} \Gamma(\lambda)(z_1,z_2) \lambda^{n-2m}   \widetilde F(\lambda r_2)  d\lambda dz_1 dz_2 .
\end{multline}
Here $i\omega_\ell^{1/2}$ has nonpositive real part and $P_{\frac{n-3}2}$ is a polynomial of degree $\frac{n-3}2$. Therefore it suffices to prove the admissibility of operators with kernel 
$$
K_{4,j}(x,y)=\int_{\R^{2n}}  \frac{v(z_1) v(z_2)\chi_{r_1 \gg \la r_2\ra}}{r_1^{n-2-j}}   \int_0^\infty  e^{c \lambda r_1 }     \chi(\lambda) \lambda^{n+j+1-2m} \Gamma(\lambda)(z_1,z_2)    \widetilde F(\lambda r_2)  d\lambda dz_1 dz_2,
$$
for $j=0,1,\ldots,\frac{n-3}2$, $|c|=1, \Re(c)\leq 0$.  This suffices to control all the terms that arise in the polynomial and for different choices of $\ell$ in \eqref{K4newdef}. Integrating by parts in $\lambda $ integral $j+3$ times we rewrite the lambda integral as  
\begin{multline*}   
- \sum_{\ell=0}^{j+2} \Big(\frac{-1}{cr_1}\Big)^{\ell+1} \partial_\lambda^{\ell}\big[  \chi(\lambda) \lambda^{n+j+1-2m} \Gamma(\lambda)   \widetilde F(\lambda r_2) \big]\big|_{\lambda=0}  \\
+ \Big(\frac{-1}{cr_1}\Big)^{j+3}  \int_0^\infty  e^{c \lambda r_1 }   \partial_\lambda^{j+3}\big[  \chi(\lambda) \lambda^{n+j+1-2m} \Gamma(\lambda)   \widetilde F(\lambda r_2) \big] d\lambda.
\end{multline*}
Note that the boundary terms are zero when $\ell <n+j+1-2m$. Since $n+j+1-2m\geq j+2$, there is a nonzero boundary term, $\ell=j+2$, only when $n=2m+1$, and it is  a constant multiple of $r_1^{-j-3}\Gamma(0)(z_1,z_2)$. The contribution of this to $K_{4,j}$ is of the form 
$$
\int_{\R^{2n}}  \frac{v(z_1) |\Gamma(0)(z_1,z_2)| v(z_2)\chi_{r_1 \gg \la r_2\ra}}{r_1^{n+1}} dz_1dz_2, 
$$
which is admissible.
We now consider the integral term.
Using the bound for $\widetilde F$ in Lemma~\ref{lem:R0cancel} and noting that $|\chi^{(k)}(\lambda)|\les 1$, we bound the integral by 
$$
\sum_{j_1+j_2+j_3\leq j+3}\int_0^1 \frac{r_2^{j_3}}{r_1^{j+3}}  \frac{  \lambda^{n+j+1-2m-j_1} |\Gamma^{(j_2)}(\lambda)| }{ \la \lambda r_2\ra^{\frac{ n-1}2+j_3}}d\lambda.
$$
Here, $j_1,j_2,j_3 \geq 0$ and $j_1\leq n+j+1-2m$. Note that the condition on $j_1$ is relevant only when $n=2m+1$. Assume first that $n\geq 2m+3$, so that $n+1-2m+j\geq j+4$.  We bound the integral by 
$$
\widetilde \Gamma(z_1,z_2)  \sum_{j_1+j_2+j_3\leq j+3}\int_0^1 \frac{1}{r_1^{j+3}}    \lambda^{j+4-j_1-j_2-j_3}   d\lambda \les \widetilde \Gamma(z_1,z_2) r_1^{-j-3} 
$$
whose contribution to $K_{4,j}$ is admissible. When $n=2m+1$, either  $j_2\geq 1$ or  $ j_3\geq 1 $. In both cases we can bound the integral by 
$$
\widetilde \Gamma(z_1,z_2)  \sum_{j_1+j_2+j_3\leq j+3}\int_0^1 \frac{\la r_2\ra}{r_1^{j+3}}    \lambda^{j+2-j_1-(j_2+j_3-1) }   d\lambda \les \widetilde \Gamma(z_1,z_2) \frac{\la r_2\ra}{r_1^{j+3}},   
$$
which has admissible contribution to $K_{4,j}$.  Hence, we conclude that the operator $K_4$ is admissible. 
 
 We now consider $K_{3}$. Using \eqref{eqn:resolv dif F}, we may write
\begin{multline*}
[\mR_0^+(\lambda^{2m})-\mR_0^-(\lambda^{2m})](z_2,y)=\frac1{\lambda^{2m-2} } [ R_0^+(\lambda^{2 })- R_0^-(\lambda^{2})](z_2,y) \\
= \frac1{\lambda^{2m-2} r_2^{n-2}}
\big[e^{i\lambda r_2} P_{\frac{n-3}2}(\lambda r_2) - e^{-i\lambda r_2} P_{\frac{n-3}2}(-\lambda r_2) \big],
\end{multline*}
and using Lemma~\ref{prop:F} for the resolvent on the left, it suffices to prove the admissibility of kernels of the form 
\begin{multline*}
K_{3,j}(x,y)=\\
\int_{\R^{2n}}  \frac{v(z_1) v(z_2)\chi_{r_2 \gg \la r_1\ra}}{r_1^{n-2m} r_2^{n-2-j}}   \int_0^\infty  \Big[e^{i \lambda (r_1 + r_2) }   -  (-1)^j e^{i \lambda (r_1 - r_2) }  \Big]   \chi(\lambda)     \lambda^{1+j}   \Gamma(\lambda)      F(\lambda r_1)  d\lambda   dz_1 dz_2,
\end{multline*}
 for $j=0,1,\ldots,\frac{n-3}2$.  In contrast to $K_{4,j}$ there is decay in both $r_1$ and $r_2$ present. Integrating by parts in $\lambda $ integral $j+3$ times we rewrite the lambda integral as 
\begin{multline*}   
 	\sum_{\ell=0}^{j+2} (-1)^{\ell} \Big[\Big(\frac{i}{ r_1+r_2}\Big)^{\ell+1} -(-1)^j \Big(\frac{i}{ r_1-r_2}\Big)^{\ell+1} \Big] \partial_\lambda^{\ell}\big[  \chi(\lambda) \lambda^{j+1} \Gamma(\lambda)    F(\lambda r_1) \big]\big|_{\lambda=0}  \\
+ (-1)^{j+3}  \int_0^\infty \Big[ \Big(\frac{i}{r_1+r_2}\Big)^{j+3} e^{i\lambda (r_1+r_2) } -  (-1)^j
 \Big(\frac{i}{r_1-r_2}\Big)^{j+3} e^{i\lambda (r_1-r_2) }\Big] \partial_\lambda^{j+3}\big[  \chi(\lambda) \lambda^{ j+1 } \Gamma(\lambda)    F(\lambda r_1) \big] d\lambda.
\end{multline*}
Once again, many of the boundary terms are zero.  The only nonzero boundary terms occur when $\ell=j+1$ or $j+2$.  When $\ell=j+2$, it is of the form
$$
 \Big[ \frac{1}{ ( r_1+r_2)^{j+3}}  -(-1)^j  \frac1{ (r_1-r_2)^{j+3} } \Big] \partial_\lambda \big[  \chi(\lambda)  \Gamma(\lambda)    F(\lambda r_1) \big]\big|_{\lambda=0}.$$
 We can bound the magnitude of this by 
 $ r_2^{-j-3}  \la r_1\ra \widetilde \Gamma(z_1,z_2), $
 whose contribution to $K_{3,j}$ is admissible as before.   On the other hand, we need to utilize cancellation for $\ell=j+1$ to see
 $$
 	\Big| \frac{1}{ ( r_1+r_2)^{j+2}}  -(-1)^j  \frac1{ (r_1-r_2)^{j+2} } \Big| =\frac{1}{r_2^{j+2}}\Big| \frac{1}{ ( 1+\frac{r_1}{r_2})^{j+2}}  -  \frac{(-1)^{2j+2}}{ (1-\frac{r_1}{r_2})^{j+2} } \Big| \les \frac{r_1}{r_2^{j+3}}.
 $$
 Hence we may bound it's contribution by $ r_2^{-j-3}  \la r_1\ra \widetilde \Gamma(z_1,z_2) $ as well.

 Using the bounds for $F$ in Lemma~\ref{prop:F}, again noting that $|\chi^{(k)}(\lambda)|\les 1$,  we bound the integral term by
 $$
\sum_{j_1+j_2+j_3 \leq j+3 } \frac{r_1^{j_3}}{ r_2^{j+3}}\int_0^1    \lambda^{ j+1 -j_1} \Gamma^{(j_2)}(\lambda)  \frac1{\la \lambda r_1\ra^{2m-\frac{n+1}2+j_3}} d\lambda.
 $$
 Here $j_1, j_2, j_3 \geq 0$ and $j_1\leq j+1$. We consider the cases $j_2=0,1$ and $j_2\geq 2$ seperately. In the former case, we bound the sum by
 $$
\widetilde \Gamma(z_1,z_2)  \sum_{  j_1 +j_3\leq j+3 } \frac{r_1^{j_3}}{ r_2^{j+3}}\int_0^1       \frac{\lambda^{ j+1 -j_1}}{\la \lambda r_1\ra^{2m-\frac{n+1}2+j_3}} d\lambda. 
 $$
 When $r_1\les 1$, this is bounded by $r_2^{-j-3} \widetilde \Gamma(z_1,z_2)$ whose contribution to $K_{3,j}$ is admissible. 
 When $r_1\gg1$, it is bounded by 
 $$
 \widetilde \Gamma(z_1,z_2) \sum_{  j_1 +j_3\leq j+3 } \frac{r_1^{j_3-j-2+j_1}}{ r_2^{j+3}}\int_0^{r_1}       \frac{\eta^{ j+1 -j_1}}{\la \eta\ra^{2m-\frac{n+1}2+j_3}} d\eta \les  \widetilde \Gamma(z_1,z_2) \frac{r_1+r_1^{ \frac{n+1}2-2m}}{ r_2^{j+3}}. 
 $$  
Here, using the   $r_1^{2m-n}$ term in $K_{3,j}$, this contribution to $K_{3,j}$ is admissible. 
In the latter case, we have the bound 
\begin{multline*}
\widetilde \Gamma(z_1,z_2)  \sum_{j_2=2}^{j+3} \sum_{j_1 +j_3\leq j+3-j_2 } \frac{r_1^{j_3}}{ r_2^{j+3}}\int_0^1    \lambda^{ j+1 -j_1-j_2+2}   \frac1{\la \lambda r_1\ra^{2m-\frac{n+1}2+j_3}} d\lambda\\
\les \widetilde \Gamma(z_1,z_2)  \sum_{j_2=2}^{j+3} \sum_{j_1 +j_3\leq j+3-j_2 } \frac{1}{ r_2^{j+3}}\int_0^1    \lambda^{ j+3 -j_1-j_2-j_3}  d\lambda \les \widetilde \Gamma(z_1,z_2)  r_2^{-j-3},
\end{multline*}
which has admissible contribution.
\end{proof}

\section{High Energy: Odd dimensions}\label{sec:high}

Since we can control the contribution of the Born series to arbitrary length, we need only consider the tail of the series in \eqref{eqn:born identity} and show that
$$
	\int_0^\infty \widetilde \chi(\lambda)[(\mathcal R_0^+ V)^\ell V\mathcal R_V^+ (V\mathcal R_0^+)^\ell V\mathcal R_0^\pm ](\lambda)\, d\lambda
$$
extend  to   bounded operators on $L^p(\R^n)$ provided $\ell$ is sufficiently large.  To do this, we invoke the limiting absorption principle established in \cite{soffernew}.  In all cases we assume there are no positive eigenvalues of $H$.  In the statement below $B(s,-s')$ is the space of bounded linear operators mapping $L^{2,s}\to L^{2,-s'}$.

\begin{theorem}[Theorem 3.9 in \cite{soffernew}]\label{thm:lap}
	
	For $k=0,1,2,3\dots,$ let $|V(x)|\les \la x \ra^{-\beta}$ for some $\beta>2+2k$, then for $s,s'>k+\frac12$, $\mR_V^{(k)}(z)\in B(s,-s')$ is continuous for $z>0$.  Furthermore,
	we have 
	$$
	\big\| \mR_V^{(k)}(z)
	\big\|_{L^{2,s}\to L^{2,-s'}}\les |z|^{\frac{1-2m}{2m}(1+k)}.
	$$
	
\end{theorem}
Note that, in particular, these bounds hold for the free resolvent.  We now collect some useful bounds on the free resolvent on high energy, when $\lambda \gtrsim 1$.  Note that throughout this section, the spectral parameter $\lambda\gtrsim 1$.  We define 
\begin{align*}
\mathcal G_x^\pm(\lambda, z)=  e^{\mp i\lambda |x|}\mR_0^\pm (\lambda^{2m})(x,z)=
\frac{e^{\pm i\lambda (|x-z|-|x|)}}{ |x-z|^{n-2m}} F(\lambda |x-y|)
\end{align*}
Following the bounds of Lemma~\ref{prop:F} and using $\lambda \gtrsim 1$, we see that
\begin{align}
|\partial_{\lambda}^\ell [\widetilde \chi(\lambda)\mathcal G_x^\pm(\lambda, z) ]|&\les \lambda^{\frac{n+1}{2}-2m} \la z_1\ra^{\ell} \bigg( \frac{1 }{|x-z|^{n-2m}} +\frac{1 }{|x-z|^{ \frac{n-1}{2} }} \bigg),  \label{eqn:G derivs}    \\
|\partial_{\lambda}^\ell [\widetilde \chi(\lambda) \mR_0^\pm(\lambda^{2m})(x,y) ]|&\les  \lambda^{\frac{n+1}{2}-2m} \bigg(\frac{1}{|x-y|^{n-2m-\ell}}+ \frac{1}{|x-y|^{ \frac{n-1}{2} -\ell}}\bigg)\label{eqn:resolv derivs hi}
\end{align}

We utilize the following fact.  It may be viewed as an extension of Lemma~3.1 in \cite{YajNew} and Lemma~2.1 in \cite{GG4wave} to higher dimensions.
\begin{lemma}\label{lem:ptwise kernel}
	Suppose that $K$ is an integral operator whose kernel obeys the pointwise bounds
	\begin{align} \label{eqn:kernels}
	|K(x,y)|\les \frac{1}{\la x\ra^{\frac{n-1}{2}} \la y \ra^{\frac{n-1}{2}} \la |x|-|y|\ra^{\frac{n+1}{2}+\epsilon}}.
	\end{align}
	Then $K$ is a bounded operator on $L^p(\mathbb R^n)$ for $1\leq p\leq \infty$ if $\epsilon>0$, and on $1<p<\infty$ if $\epsilon=0$.  
	
\end{lemma}

\begin{prop}\label{prop:high odd}
	
	We have the bound 
	\begin{multline}\label{eqn:bs tail int}
	\bigg|\int_0^\infty \widetilde \chi(\lambda)\lambda^{2m-1}
	(\mathcal R_0^+(\lambda^{2m}) V)^{\ell +1}  \mathcal R_V^+(\lambda^{2m})V (\mathcal R_0^+(\lambda^{2m})V)^{\ell } \mathcal R_0^\pm (\lambda^{2m})(x,y)\, d\lambda\bigg|\\
	\les \frac{1}{\la |x|-|y|\ra^{\frac{n+3}{2}} \la x\ra^{\frac{n-1}{2} } \la y \ra^{\frac{n-1}{2}}},
	\end{multline}
	provided $\ell $ is sufficiently large, and $|V(x)|\les \la x\ra^{-\beta}$ for some $\beta>n+5$.
	In particular, this kernel is admissible and hence the tail extends to a bounded operator on $L^p(\R ^n)$ for all $1\leq p\leq \infty$.
	
\end{prop}

\begin{proof}
	
	We first establish the boundedness of the integral.  We note that for $\sigma>\f12$ and $\ell_1 =\lfloor \frac{n}{4m} \rfloor+1$ we have
	\begin{align}\label{eqn:iter res1}
	\|(V\mathcal R_0^+)^{\ell_1-1}V \mathcal R_0^\pm (\lambda^{2m})(\cdot, y)\|_{L^{2,\sigma}}  \les \frac{\lambda^{\ell_1(\frac{n+1}{2}-2m)}}{\la y\ra^{\frac{n-1}{2}}}.
	\end{align}
	This follows using the representations \eqref{eqn:resolv derivs hi} with $\ell=0$ and Lemma~\ref{lem:IMRN lems} repeatedly as in Lemma~\ref{lem:A lems}:
	After $\ell_1 =\lfloor \frac{n}{4m} \rfloor+1$ iterations of Lemma~\ref{lem:IMRN lems} in the spatial variables $z_1,z_2,\dots$, we arrive at a bound for the kernel of the operator in \eqref{eqn:iter res1}.  This bound is  dominated by $|y-z_j|^{-(\frac{n-1}{2})}$, which is locally $L^2(\R^n)$.  By Lemma~\ref{lem:GV lem}, we may bound \eqref{eqn:iter res1} by
	$$
		\lambda^{\ell_1(\frac{n+1}{2}-2m)}\| \la z_j\ra^{-\beta }|y-z_j|^{-(\frac{n-1}{2})} \|_{L^{2,\sigma}} \les \frac{\lambda^{\ell_1(\frac{n+1}{2}-2m)}}{\la y \ra^{\frac{n-1}{2}}},
	$$
	provided that $\beta> \sigma+ \frac{n}{2}$.  
	Similarly, 
	\begin{align}\label{eqn:iter res2}
	\|\mathcal  (\mathcal R_0^+V)^{\ell_1 }(x, \cdot)\|_{L^{2,\sigma}}  \les \frac{\lambda^{\ell_1(\frac{n+1}{2}-2m)}}{\la x\ra^{\frac{n-1}{2}}}.
	\end{align}
	By repeated uses of Theorem~\ref{thm:lap}, we see that
	\begin{align}\label{eqn:lap rep}
	\|  (\mathcal R_0^+ V)^{\ell_2}\mathcal R_V^+(V \mathcal R_0^+)^{\ell_2}\|_{L^{2,\sigma}\to L^{2,-\sigma}} \les \lambda^{(2\ell_2+1)(1-2m)}.
	\end{align}
	Let $\ell=\ell_1+\ell_2$, then combining \eqref{eqn:iter res1}, \eqref{eqn:iter res2} and \eqref{eqn:lap rep} we see that
	\begin{multline}\label{eqn:hi tail noibp}
	\bigg|\int_0^\infty \widetilde \chi(\lambda)\lambda^{2m-1}
	(\mathcal R_0^+(\lambda^{2m}) V)^{\ell_1 +\ell_2}  \mathcal R_V^+(\lambda^{2m})V (\mathcal R_0^+(\lambda^{2m})V)^{\ell } \mathcal R_0^\pm (\lambda^{2m})(x,y)\, d\lambda\bigg|\\
	=\int_0^\infty \widetilde \chi(\lambda)\lambda^{2m-1} \|(\mathcal R_0^+V)^{\ell_1+\ell_2 }(x, \cdot)\|_{L^{2,+\f12+}} \| (\mathcal R_0^+V)^{\ell_2 }\mathcal R_V^+(V\mathcal R_0^+)^{\ell_2}\|_{L^{2,\f12+}\to L^{2,-\f12-}} \\
	\times  \|(V\mathcal R_0^+)^{\ell_1-1} V\mathcal R_0^\pm  (\lambda^{2m})(\cdot, y)\|_{L^{2,\f12+}}\, d\lambda\\ 
	\les \frac{1}{ \la x\ra^{\frac{n-1}{2} } \la y \ra^{\frac{n-1}{2}}} \int_1^\infty \lambda^{\ell_1(n+1-4m)+(2\ell_2+1)(1-2m)} \,d\lambda 
	\les \frac{1}{ \la x\ra^{\frac{n-1}{2} } \la y \ra^{\frac{n-1}{2}}}.
	\end{multline}
	By selecting $\ell_2$ large enough, the $\lambda$ integral converges.
	To complete the proof, we use the functions $\mathcal G^\pm$   and integrate by parts $\frac{n+3}{2}$ times.   That is,
	\begin{multline*}
	\int_0^\infty \widetilde \chi(\lambda)\lambda^{2m-1}
	(\mathcal R_0^+(\lambda^{2m}) V)^{\ell +1}  \mathcal R_V^+(\lambda^{2m})V (\mathcal R_0^+(\lambda^{2m})V)^{\ell }  R_0^{\pm} (\lambda^{2m})(x,y)\, d\lambda \\
	=   \int_0^\infty e^{-i\lambda(|x|\pm |y|)}
	\widetilde \chi(\lambda)\lambda^{2m-1}  \mathcal G_x^+(\lambda, z_1)V(z_1)
	(\mathcal R_0^+(\lambda^{2m}) V)^{\ell  }  \mathcal R_V^+(\lambda^{2m})
	V (\mathcal R_0^+(\lambda^{2m})V)^{\ell }  \mathcal G^\pm_y(\lambda, z_{2\ell +1})\, d\lambda\\
	=\bigg(\frac{-1}{i(|x|\pm|y|)}\bigg)^{\frac{n+3}{2}}\int_0^\infty e^{-i\lambda(|x|\pm |y|)} \partial_\lambda^{\frac{n+3}{2}} \bigg(  \widetilde \chi(\lambda)\lambda^{3-2m}  \mathcal G^+_x(\lambda, \cdot)V(\cdot)
	(\mathcal R_0^+(\lambda^{2m}) V)^{\ell  } \\ 
	\mathcal R_V^+(\lambda^{2m})
	V (\mathcal R_0^+(\lambda^{2m})V)^{\ell }  \mathcal G^\pm_y(\lambda, \cdot)\bigg) \, d\lambda.
	\end{multline*}
	By the limiting absorption principle and the support of $\widetilde\chi(\lambda)$, there are no boundary terms when integrating by parts.  To complete the argument, let $k_j\in \mathbb N\cup\{0\}$ be such that $\sum k_j=\frac{n+3}{2}$, then the contribution will be bounded by
	\begin{multline*}
	\frac{1}{|\,|x|-|y|\,|^{\frac{n+3}{2}}}\int_0^\infty   | \widetilde \chi(\lambda)\lambda^{3-2m-k_1}   |\partial_\lambda^{k_2}\mathcal G^+_x (\lambda,\cdot)V| \,|
	\partial_\lambda^{k_3}(\mathcal R_0^+(\lambda^{2m}) V)^{\ell_1 }| \\ 
	|\partial_\lambda^{k_4}[(\mathcal R_0^+(\lambda^{2m}) V)^{\ell_2 } \mathcal R_V^+(\lambda^{2m})(V\mathcal R_0^+(\lambda^{2m}))^{\ell_2 }]|
	\,| 
	\partial_\lambda^{k_5}(V\mathcal R_0^+(\lambda^{2m}))^{\ell_1 } V| |\partial_\lambda^{k_6}\mathcal G^\pm_y (\lambda,\cdot)|\,|
	\, d\lambda.
	\end{multline*}
	Invoking the bounds in \eqref{eqn:resolv derivs hi} and an argument similar to the first case shows that we have the bound
	\begin{multline*}
	\frac{1}{|\,|x|-|y|\,|^{\frac{n+3}{2}}}\int_0^\infty    \widetilde \chi(\lambda)\lambda^{2m-1-k_1}   \||\partial_\lambda^{k_2}\mathcal G^+_x (\lambda,\cdot)V|\,|
	\partial_\lambda^{k_3}(\mathcal R_0^+(\lambda^{2m}) V)^{\ell_1 }\|_{L^{2,\f12+k_4+  }} \\
	\| \partial_\lambda^{k_4}[\mathcal R_0^+(\lambda^{2m}) V)^{\ell_2 } \mathcal R_V^+(\lambda^{2m})(V\mathcal R_0^+(\lambda^{2m}))^{\ell_2 }]\|_{L^{2,-\f{1}{2}-k_4-}\to L^{2,-\f{1}{2}-k_4-}}\\  \||
	\partial_\lambda^{k_5}((V\mathcal R_0^+(\lambda^{2m}))^{\ell_1 }V| |\partial_\lambda^{k_6}\mathcal G^\pm_y (\lambda,\cdot)|\,|\|_{L^{2,\f12+k_4+}}\, d\lambda\\
	\les \frac{1}{|\,|x|-|y|\,|^{\frac{n+3}{2}} \la x\ra^{\frac{n-1}{2} } \la y \ra^{\frac{n-1}{2}}}.
	\end{multline*}
	We note that the decay rate of $|V(z)|\les \la z\ra^{-(n+5)-}$ is necessitated when all derivatives act on $\mR_V$ to apply the limiting absorption principle, Theorem~\ref{thm:lap}.  This suffices to control the other extreme cases, when all derivatives act on a single free resolvent, then by \eqref{eqn:G derivs}, \eqref{eqn:resolv derivs hi} and Lemma~\ref{lem:GV lem} this decay rate on $V$ suffices to push forward decay in $x$ or $y$ respectively.   
	Combining this with \eqref{eqn:hi tail noibp} establishes the desired bound.  Invoking Lemma~\ref{lem:ptwise kernel} establishes the claim on $L^p$ boundedness.

\end{proof}

By integrating by parts one less time, one obtains the following which requires less decay of the potential but fails to capture the endpoints.

\begin{corollary}
	
	We have the bound 
	\begin{multline} 
	\bigg|\int_0^\infty \widetilde \chi(\lambda)\lambda^{2m-1}
	(\mathcal R_0^+(\lambda^{2m}) V)^{\ell +1}  \mathcal R_V^+(\lambda^{2m})V (\mathcal R_0^+(\lambda^{2m})V)^{\ell } \mathcal R_0^\pm (\lambda^{2m})(x,y)\, d\lambda\bigg|\\
	\les \frac{1}{\la |x|-|y|\ra^{\frac{n+1}{2}} \la x\ra^{\frac{n-1}{2} } \la y \ra^{\frac{n-1}{2}}},
	\end{multline}
	provided $\ell $ is sufficiently large and $|V(x)|\les \la x\ra^{-\beta}$ for some $\beta>n+3$.
	In particular, this kernel   extends to a bounded operator on $L^p(\R ^n)$ for all $1< p< \infty$.
	
\end{corollary}

\section{Even dimensions}\label{sec:even}

In this section we show how the low and high energy results for the tail of the Born series in odd dimensions proven in Sections~\ref{sec:low} and \ref{sec:high} may be applied to even dimensions.  One requires minor modifications to account for the logarithmic singularities of the resolvent.  After developing an appropriate representation of the free resolvent in Lemma~\ref{propFeven}, the arguments may be easily adapted.

First we sketch the argument for low energies. We will prove 
\begin{prop}\label{prop:low energyeven}
	Let $n>2m$ be even.
	If $|V(x)|\les \la x\ra^{-\beta}$ for some $\beta>n+3$, then the operator defined by
	$$
	-\frac{m}{\pi i} \int_0^\infty \chi(\lambda) \lambda^{2m-1}\mR_0^+(\lambda^{2m})vM^+(\lambda)^{-1}v[\mR_0^+-\mR_0^-](\lambda^{2m})\, d\lambda
	$$
	extends to a bounded operator on $L^p(\R^n)$ for all $1<p<\infty$.
\end{prop}

We have the following representation for the even dimensional free resolvent.

\begin{lemma}\label{propFeven}
	
	Let $n>2m$ be even. Then,   we have the following representation of the free resolvent
	$$
		\mR_0^+(\lambda^{2m})(y,u)
	=  \frac{e^{i\lambda|y-u|}}{|y-u|^{n-2m}} F(\lambda |y-u| ).$$
	Here   $|F^{(N)}(r)|\les  \la r\ra^{\frac{n+1}2 -2m-N}$, $N=0,1,2,...,2m-1$, and is valid for any $N$ when $r\gtrsim 1$, while when $r\ll 1$ we have $|F^{(2m)}(r)|\les \log(r)$ and $|F^{(N)}(r)|\les r^{2m-N}$ for $N>2m$.
	
\end{lemma}

\begin{proof}
	
	To prove this we consider cases when $\lambda |y-u|\ll1$ and $\lambda|y-u|\gtrsim 1$.  We consider first the second-order Schr\"odinger resolvent, which may be expressed in terms of the Bessel functions
	$$
	R_0^+(\lambda^{2})(y,u) =\frac{i}{4}\bigg(\frac{\lambda}{2\pi |y-u|}\bigg)^{\frac{n-2}{2}} H^{(1)}_{\frac{n-2}{2}}(\lambda |y-u|).
	$$
	Unlike in odd dimensions, we do not have a closed form representation for the Hankel function of the first kind $H^{(1)}_{\frac{n-2}{2}}(\cdot)$.  Following the approach in \cite{GG2}, see also \cite{Jen},  for $\lambda |y-u|\ll1$, we have a series of the form
	$$
	R_0^+(\lambda^{2})(y,u)=\frac{1}{|y-u|^{n-2}} \sum_{j=0}^\infty \sum_{k=0}^1   c_j (\lambda |y-u|)^{2j} (a_{ j} \log (\lambda|y-u|)+b_{ j} )^k.
	$$
	The constants $a_j,b_j,c_j$ may be computed explicitly.  Of particular importance is that $a_j=0$ for $j\leq \frac{n}{2}-2$.
	Combining this with the splitting identity \eqref{eqn:Resolv}, we have
	\begin{multline*}
	\mR_0^+(\lambda^{2m})(y,u)=\frac{1}{m|y-u|^{n-2}\lambda^{2m-2}}  \sum_{j=0}^\infty \sum_{k=0}^1 \sum_{\ell=0}^{m-1}  c_j \omega_\ell^{j+1} (\lambda |y-u|)^{2j} (a_{ j} \log (\lambda|y-u|)+a_j\log(\omega_\ell) +b_{ j} )^k.
	\end{multline*}
	Using the fact that
	$$
	\sum_{\ell=0}^{m-1} \omega_\ell^{j+1}\neq 0 \quad \text{ if and only if } \quad j=km-1, \quad k=1,2,\dots,
	$$
	we may write (for $\lambda |y-u|\ll 1$)
	$$
	\mR_0^+(\lambda^{2m})(y,u)=\frac{1}{|y-u|^{n-2m}}\bigg(\sum_{j=0}^{m-1} d_j (\lambda |y-u|)^{2j} + \sum_{j=m}^{\infty} d_j (\lambda |y-u|)^{2j} (1+d_{j,l} \log (\lambda |y-u|))  \bigg)
	$$
	In particular,
	the first logarithm occurs at the term $\lambda^{2m}$.  The claim on $F$ for $r\ll 1$ follows since we may write, for any choice of $N$
	$$
		F(r)=e^{-ir}\bigg(\sum_{j=0}^{m-1} d_j r^{2j} + \sum_{j=m}^{N} d_j (r^{2j} (1+d_{j,l} \log (r))  \bigg)+O(r^{N-})
	$$
	where the remainder may be differentiated arbitrarily many times.

	The large argument expansion of the resolvent is the same from the Bessel functions, one has for $\lambda |y-u|\gtrsim 1$ that
	$$
		\lambda^{2-2m}R_0^+ (\lambda^2)(y,u)=e^{i\lambda|y-u|} \frac{(\lambda|y-u|)^{\frac{n+2}{2}-2m}}{|y-u|^{n-2m}} \omega_+(\lambda |y-u|),
	$$
	where $|\partial_r^k\omega_+(r)|\les r^{-\f12-k}$.  The splitting identity \eqref{eqn:Resolv} along with the exponential decay of the other resolvents suffices to establish the claim.
	
\end{proof}

\begin{lemma}\label{lem:low resolv sup even}
	Let $n>2m $ be even. 
	We have the following bounds on the derivatives of the resolvent.  For $k=1,2,\dots,$ we have
	$$
	\sup_{0<\lambda <1} |\lambda^{k-1}\partial_\lambda^k \mR_0(\lambda^{2m})(x,y)| \les 
	|x-y|^{2m+1-n}+|x-y|^{k-(\frac{n-1}{2})}.
	$$
\end{lemma} 
\begin{proof}
Since $|\partial_r^k F(r)|$ satisfies the same bounds as in the odd case for $k\leq 2m-1$ and for $r>1$, we can assume that $k\geq 2m$ and $\lambda|x-y|<1$.  We have
\begin{multline*}
\lambda^{k-1}|\partial_\lambda^k \mR_0(\lambda^{2m})(x,y)|\les \lambda^{k-1}|x-y |^{k+2m-n} \sum_{\ell=0}^k |F^{(\ell)}(\lambda |x-y|)|\\
\les\lambda^{k-1}|x-y|^{k+2m-n} \big[1+|\log(\lambda |x-y|)|+(\lambda |x-y|)^{2m-k}\big]
\\\les |x-y|^{1+2m-n} [\lambda|x-y|]^k \big[(\lambda|x-y|)^{0-}+(\lambda |x-y|)^{2m-k}\big] \les  |x-y|^{1+2m-n}.
\end{multline*}
\end{proof}
With this the invertibility of $M(\lambda)$ and the bounds on its derivatives follow by similar arguments to the odd dimensional case, namely 
\be\label{eq:MGammaeven}
\Gamma_N(x,y):=\sup_{0<\lambda<\lambda_0}\lambda^N | \partial_\lambda^N[M^+(\lambda)]^{-1}(x,y)|
\ee
is bounded in $L^2$ for $N=0,1,\ldots,\frac{n+2}2$, provided that $\beta >n+3$.  This will suffice for $n<4m$ even. For $n\geq 4m$ even, we iterate the Born series and note that $A(\lambda, z_1,z_2)$ defined via \eqref{eq:Alambda} satisfies a slightly modified  version  of the claim of Lemma~\ref{lem:A lems}:   
$$	\sup_{0<\lambda <1}|\lambda^\ell \partial_\lambda^\ell 	A(\lambda, z_1,z_2)| \les \la z_1 \ra^{\frac32}  \la z_2\ra^{\frac32},
$$
	for $0\leq \ell\leq \frac{n+2}{2}$. The inclusion of $\lambda^\ell$ power takes care of the singularity arising from the logarithm in Lemma~\ref{propFeven} as in Lemma~\ref{lem:low resolv sup even}. Therefore letting  
$$\Gamma=w A(\lambda)v M^{-1}(\lambda)vA(\lambda) w,$$
as above, we see that  
\be\label{eq:tildegammaeven}
\widetilde \Gamma (x,y):= \sup_{0<\lambda <\lambda_0} \sup_{0\leq k\leq  \frac{n+2}2} \big|\lambda^k \partial_\lambda^k \Gamma(\lambda)(x,y)\big| \les \la x\ra^{-\frac{n}2-}\la y\ra^{-\frac{n}2-},
\ee
provided that $\beta>n+3$.  The following variant of Lemma~\ref{lem:low tail low d} finishes the proof:  
\begin{lemma}\label{lem:low tail low d even} Fix $n>2m$ even and let $\Gamma$ be a $\lambda $ dependent absolutely bounded operator. Assume that   
$$
\widetilde \Gamma (x,y):= \sup_{0<\lambda <\lambda_0} \sup_{0\leq k\leq  \frac{n+2}2} \big|\lambda^k \partial_\lambda^k \Gamma(\lambda)(x,y)\big| 
$$
  is bounded on $L^2$ for $2m<n<4m$ and satisfies \eqref{eq:tildegammaeven} for $n\geq 4m$. 
   Then the operator with kernel 
	$$
	K(x,y)=\int_0^\infty \chi(\lambda) \lambda^{2m-1} [\mR_0^+(\lambda^{2m}) v \Gamma v \mR_0^-(\lambda^{2m})](x,y)   d\lambda 
	$$
	is bounded on $L^p$ for $1<p<\infty$ provided that $\beta>n$.
\end{lemma}
\begin{proof} Writing 
$$
K(x,y)=\int_{\R^{2n}} \frac{v(z_1)v(z_2)}{r_1^{n-2m}r_2^{n-2m}}\int_0^\infty  e^{i\lambda(r_1-r_2)}\chi(\lambda)\lambda^{2m-1}\Gamma(\lambda) F(\lambda r_1)F(\lambda r_2) d\lambda dz_1dz_2,
$$
we see that the $\lambda $ integral satisfies the bound \eqref{kbound0}:
$$
	 \widetilde\Gamma (z_1,z_2)    \int_0^1       \frac{\lambda^{2m-1}   }{ \la \lambda r_1\ra^{2m-\frac{n+1}2} \la \lambda r_2\ra^{2m-\frac{n+1}2} } d\lambda.
$$
We will use this for $\lambda |r_1-r_2|<1$ and integrate  by parts $N\leq \frac{n+2}{2}$ times otherwise. Note that by Lemma~\ref{propFeven}, when $\lambda r >1$ or when $\ell\leq 2m-1$, we have  $|\lambda^\ell \partial_\lambda^\ell F(\lambda r )| \les \la r \ra^{\frac{n+1}2 -2m}$. When $\lambda r <1$ and  $\ell\geq  2m $, we once again have 
$$|\lambda^\ell \partial_\lambda^\ell F(\lambda r )| \les    (\lambda r_1)^{\ell} \big(|\log(\lambda r )|+(\lambda r )^{2m-\ell}\big)\les 1\les \la \lambda r \ra^{\frac{n+1}2 -2m}.$$
Therefore, we obtain the following bound essentially identical to \eqref{kbound02}:
	\be 
	\label{kbound02even}
	|K(x,y)|\les \int_{\R^{2n}} \frac{v(z_1)  \widetilde\Gamma (z_1,z_2) v(z_2)}{  r_1^{n-2m} r_2^{n-2m}} \int_0^1 \frac{ \lambda^{2m-1}    } { \la \lambda (r_1-r_2)\ra^{N}   \la \lambda r_1\ra^{2m-\frac{n+1}2}  \la \lambda r_2\ra^{2m-\frac{n+1}2} } d\lambda  dz_1 dz_2,
	\ee
for all $0\leq N\leq \frac{n+2}2$, noting $N$ need not be an integer. The rest of the proof is identical to the proof of Lemma~\ref{lem:low tail low d} using \eqref{kbound02even} with $N=0$ for $K_1$ with $N=2$ for $K_2$ and $N=\frac{n+1}2$ for $K_3$ and $K_4$.
\end{proof}

\begin{prop}\label{prop:hi even}
	
	We have the bound 
	\begin{multline*} 
	\bigg|\int_0^\infty \widetilde \chi(\lambda)\lambda^{2m-1}
	(\mathcal R_0^+(\lambda^{2m}) V)^{\ell +1}  \mathcal R_V^+(\lambda^{2m})V (\mathcal R_0^+(\lambda^{2m})V)^{\ell } \mathcal R_0^\pm (\lambda^{2m})(x,y)\, d\lambda\bigg|\\
	\les \frac{1}{\la |x|-|y|\ra^{\frac{n+2}{2}} \la x\ra^{\frac{n-1}{2} } \la y \ra^{\frac{n-1}{2}}},
	\end{multline*}
	provided $\ell $ is sufficiently large, and $|V(x)|\les \la x\ra^{-\beta}$ for some $\beta>n+4$.
	In particular, this kernel is admissible and hence the tail extends to a bounded operator on $L^p(\R ^n)$ for all $1\leq p\leq \infty$.
	
\end{prop}

This proof is essentially identical to the proof of Proposition~\ref{prop:high odd} in the odd dimensional case.  Here, by Lemma~\ref{propFeven}, the bounds \eqref{eqn:G derivs} and \eqref{eqn:resolv derivs hi} hold, hence the proposition follows by integrating by parts $\frac{n+2}{2}$ times to invoke Lemma~\ref{lem:ptwise kernel}.

\section{Integral estimates and Proofs of Technical Lemmas}\label{sec:int ests}

We now present the proofs of some technical lemmas that are used throughout the paper.  For completeness we provide a proof of Lemma~\ref{lem:ptwise kernel}.

\begin{proof}[Proof of Lemma~\ref{lem:ptwise kernel}]	
	We first consider the case when $\epsilon = 0$, we decompose the integral into three regions according to whether $|x| > 2|y|$, $|x| <  \frac12 |y|$, or $\frac12|y| \leq |x| \leq 2|y|$.  In the region where $|x|\approx |y|$,  switching to polar coordinates we see that
	$$
	\int_{|x| \approx |y|}|K(x,y)|dx \les  \frac{1}{\la y\ra^{n-1}}\int_{ |y|/2}^{2|y|} \frac{r^2}{\la r-|y|\ra^{\frac{n+1}{2}}}\, dr \les \frac{|y|^{n-1}}{\la y\ra^{n-1}}\int_{ |y|/2}^{2|y|} \frac{1}{\la r-|y|\ra^{\frac{n+1}{2}}}\, dr\les 1,
	$$ 
	uniformly in $y$. By symmetry in $x$ and $y$, this part of the operator has an admissible kernel and is bounded for any $1 \leq p \leq \infty$.
	
	On the second region (using that $|\,|x|-|y|\,| \approx |y|$ when $|x|<\f12|y|$) we see
	\[
	\int_{|x| < \frac12|y|} \frac{1}{\la x\ra^{\frac{n-1}{2}p}\la y\ra^{\frac{n-1}{2}p} \la|x| - |y|\ra^{\frac{n+1}{2}p}}dx \les \int_{|x| < \frac12|y|} \frac{1}{\la x\ra^{\frac{n-1}{2}p} \la y\ra^{np}}dx \les \la y\ra^{\max(n -\frac{3n-1}{2}p, -np)},
	\] 
	and (using that $|\,|x|-|y|\,| \approx |x|$ when $|x|>2|y|$)
	\[
	\int_{|x| > 2|y|} \frac{1}{\la x\ra^{\frac{n-1}{2}p} \la y\ra^{\frac{n-1}{2}p}\la|x| - |y|\ra^{\frac{n+1}{2}p}}dx \les \int_{|x| > 2|y|} \frac{1}{\la x\ra^{np}\la y\ra^{\frac{n-1}{2}p}}dx \les \la y\ra^{n-\frac{3n-1}{2}p} \text{ when } p>1.
	\] 
	The constraint on the range of $p$ occurs when $|x|$ is large.  Noting that $\la y\ra^{\max(n/p -\frac{3n-1}{2}, -n)} = \la y\ra^{\max(-n/p' - (n-1)/2, -n)}$ belongs to $L^{p'}(\R^n)$ for any $1< p < \infty$ so these parts of the operator $K(x,y)$ are bounded on $L^p(\R^n)$ as long as $1 < p < \infty$.
	
	When $\epsilon > 0$ using polar coordinates we see that
	\[
	\sup_{y\in \mathbb R^n}\int_{\R^n} \frac{1}{\la x\ra^{\frac{n-1}{2}} \la y \ra^{\frac{n-1}{2}} \la |x|-|y|\ra^{\frac{n+1}{2}+\epsilon}}\,dx
	= \sup_{y\in \mathbb R^n} C_n \int_0^\infty \frac{r^{n-1}}{\la r\ra^{\frac{n-1}{2}} \la y\ra^{\frac{n-1}{2}} \la r-|y|\ra^{\frac{n+1}{2}+\epsilon}}dr
	\les 1.
	\]
	The last inequality follows by  breaking up into regions based on whether $r\leq \f12|y|$, $r\approx |y|$ or $r\geq 2|y|$.  Similar to the previous case, integrability for large $r$ requires $\epsilon >0$.  By symmetry in $x$ and $y$, $K$ has an admissible kernel and is bounded for $1\leq p\leq \infty$.
	
\end{proof}
 
Finally, the following elementary integral estimates are used throughout the paper.

\begin{lemma}[Lemma 3.8 in \cite{GV}]\label{lem:GV lem}
	
	Let $k, \beta$ be such that $k<n$ and $n<\beta+k$.  Then
	$$
	\int_{\R^n} \frac{du}{\la u\ra^\beta |x-u|^k}\les \left\{  \begin{array}{ll}
	\la x\ra^{n-\beta-k} & \beta<n\\
	\la x \ra ^{-k} & \beta >n
	\end{array} \right..
	$$
	
\end{lemma}

\begin{lemma}[Lemma~6.3 in \cite{EG}]\label{lem:IMRN lems}
	
	Fix $u_1, u_2 \in \R^n$, and let $0\leq k, \ell<n$, $\beta > 0$,  
	$k+\ell+\beta \geq n$, $k+\ell \neq n$.   We have 
	\begin{align*}
	\int_{\R^n}\frac{\langle z \rangle^{-\beta-} dz}
	{|z-u_1|^k|z-u_2|^\ell} \lesssim \left\{
	\begin{array}{lc} 
	\big(\frac{1}{|u_1-u_2|}\big)^{\max(0,k+\ell-n)} 
	& |u_1-u_2|\leq 1\\
	\big(\frac{1}{|u_1-u_2|}\big)^{\min(k, \ell, k+\ell+\beta-n)} 
	& |u_1-u_2| > 1
	\end{array}\right..
	\end{align*}
	
\end{lemma}

\section*{Acknowledgments}
The authors wish to thank the anonymous referee whose thorough review greatly improved the presentation of this paper.


\begin{thebibliography}{9}
	
	
	\bibitem{agmon} S. Agmon, {\em Spectral properties of Schr\"odinger operators and scattering theory.} 	Ann.\ Scuola Norm.\ Sup.\ Pisa Cl.\ Sci. (4) 2 (1975), no.~2, 151--218.
	
	\bibitem{Bec}
	M. Beceanu, \emph{Structure of wave operators for a scaling-critical class of potentials}. Amer. J. Math. 136 (2014), no. 2, 255--308.
	
	\bibitem{BS}
	M. Beceanu, and W. Schlag, \emph{Structure formulas for wave operators.}  Amer. J. Math. 142 (2020), no. 3, 751--807.
	
	\bibitem{BS2}
	M. Beceanu, and W. Schlag,
	\emph{Structure formulas for wave operators under a small scaling invariant condition.}  J. Spectr. Theory 9 (2019), no. 3, 967–990.
	

	
	\bibitem{DF} P. D'Ancona, and L. Fanelli, \emph{Lp-boundedness of the wave operator for the one dimensional Schr\"odinger operator}, Commun. Math. Phys. 268 (2006), 415--438. 
	

	\bibitem{EG} M.~B. Erdo\smash{\u{g}}an,  and  W.~R.~Green,	\emph{Dispersive estimates for the Schr\"odinger equation for $C^{\frac{n-3}{2}}$ potentials in odd dimensions,} Int. Math. Res. Notices 2010:13, 2532--2565.
	
	\bibitem{egt} M.~B. 
	Erdo\smash{\u{g}}an, W.~R.~Green, and E. Toprak, \emph{On the Fourth order Schr\"odinger equation in three dimensions: dispersive estimates and zero energy resonances}.   J. Differ. Eq., 267, (2019), no. 3, 1899--1954.

	
	\bibitem{fsy}
	H. Feng, A. Soffer, and X. Yao, \emph{Decay estimates and Strichartz estimates of fourth order Schr\"odinger operator}.  Journal of Functional Analysis, Volume 274, Issue 2, 2018, 605--658.
	
	
	\bibitem{soffernew} H. Feng, A. Soffer, Z. Wu, and X. Yao, {\it Decay estimates for higher order elliptic operators},  Trans. Amer. Math. Soc. 373 (2020), no. 4, 2805–-2859. 
	
	\bibitem{GG1} M. Goldberg and W. Green,  \emph{Dispersive Estimates for higher dimensional Schr\"odinger Operators with threshold eigenvalues I: The odd dimensional case,} J. Funct. Anal. 269 (2015) no. 3, 633--682.
	
	\bibitem{GG2} M. Goldberg, and W. Green, \emph{Dispersive Estimates for higher dimensional Schr\"odinger Operators with threshold eigenvalues II: The even dimensional case,} J. Spectr. Theory 7 (2017), 33--86. 

	\bibitem{GGwaveop}
	M. Goldberg and W. Green, 
	\emph{The $L^p$ boundedness of wave operators for Schr\"odinger operators with threshold singularities}.   Adv. Math. 303 (2016), 360--389. 

\bibitem{GG4wave} M. Goldberg, and W. Green, \emph{On the $L^p$ boundedness of the Wave Operators for fourth order Schr\"odinger operators,}. Trans. Amer. Math. Soc. 374 (2021),  4075--4092.

\bibitem{GV}
M. Goldberg, and M.~Visan.  \emph{A Counterexample to Dispersive Estimates}.
Comm. Math. Phys. 266 (2006), no. 1, 211-238. 



\bibitem{GT4}  W. Green, and E. Toprak,
\emph{On the Fourth order Schr\"odinger equation in four dimensions: dispersive estimates and zero energy resonances,} J. Differential Equations, 267, (2019), no. 3, 1899--1954.

\bibitem{Hor}L. H\"ormander, \emph{The existence of wave operators in scattering theory.} Math. Z. 146 (1976), no. 1, 69--91. 



\bibitem{Jen}
A. Jensen,  \emph{Spectral properties of Schr\"odinger operators and time-decay of the wave functions results in $L^2(R^m)$,
	$m\geq 5$}. Duke Math.\ J.\ 47 (1980), no. 1, 57--80.
	
	\bibitem{JY2}
	A. Jensen, and K. Yajima,
	\emph{A remark on $L^p$-boundedness of wave operators for two-dimensional Schr\"odinger operators.}
	Comm. Math. Phys. 225 (2002), no. 3, 633--637. 
	
	\bibitem{JY4}
	A. Jensen, and K. Yajima, 
	\emph{On $L^p$ boundedness of wave operators for 4-dimensional Schr\"odinger operators with threshold singularities.}
	Proc. Lond. Math. Soc. (3) 96 (2008), no. 1, 136--162. 
	
	\bibitem{Kur1} S. Kuroda, \emph{Scattering theory for differential operators. I.}  J. Math. Soc. Japan 25 (1973), 75--104.
	
	\bibitem{Kur2} S. Kuroda, \emph{Scattering theory for differential operators. II. Self-adjoint elliptic operators.} J. Math. Soc. Japan 25 (1973), 222--234.
	

	
	\bibitem{Miz} H. Mizutani,   \emph{Wave operators on Sobolev spaces.} Proc. Amer. Math. Soc. 148 (2020), no. 4, 1645--1652.
	
	\bibitem{MWY} H. Mizutani, Z. Wan, and X. Yao, \emph{$L^p$-boundedness of wave operators for fourth-order Schr\"odinger operators on the line}, preprint, 2022.   	arXiv:2201.04758
	

	\bibitem{Sche} M. Schechter,
	\emph{Scattering theory for pseudodifferential operators,}
	Quart. J. Math. Oxford Ser. (2) 27 (1976), no. 105, 111--121. 
	
	\bibitem{ScheArb} M. Schechter,
	\emph{Scattering theory for elliptic operators of arbitrary order.} Comment. Math. Helv. 49 (1974), 84--113.
	

	
	\bibitem{YajWkp1}
	K. Yajima, \emph{The $W^{k,p}$-continuity of wave operators for Schr\"odinger operators.}
	J. Math. Soc. Japan 47 (1995), no. 3, 551--581. 
	
	\bibitem{YajWkp2}
	K. Yajima, \emph{The $W^{k,p}$-continuity of wave operators for Schr\"odinger operators. II. Positive potentials in even dimensions $m\geq 4$.} Spectral and scattering theory (Sanda, 1992), 287--300,
	Lecture Notes in Pure and Appl. Math., 161, Dekker, New York, 1994. 
	
	\bibitem{YajWkp3}
	K. Yajima, \emph{The $W^{k,p}$-continuity of wave operators for Schr\"odinger operators. III. Even-dimensional cases $m\geq 4$.}
	J. Math. Sci. Univ. Tokyo 2 (1995), no. 2, 311--346. 
	

	
	\bibitem{Yaj}
	K. Yajima, \emph{The $L^p$ Boundedness of wave operators for Schr\"{o}dinger
		operators with threshold singularities I. The odd dimensional case}. J.\ Math.\
	Sci.\ Univ.\ Tokyo 13 (2006), 43--94.
	
	\bibitem{YajNew}	K. Yajima, \emph{Wave Operators for Schr\"odinger Operators with Threshold Singularities, Revisited}.  Preprint, 	arXiv:1508.05738.

	\bibitem{YajNew2}
	K. Yajima, \emph{Remark on the
		$L^p$-boundedness of wave operators for Schr\"odinger operators with threshold singularities}, Documenta Mathematica
	21
	(2016), 391--443.
	
	\bibitem{YajNew3}
	K. Yajima,  \emph{On wave operators for Schr\"odinger operators with threshold singularities in three dimensions.}   Tokyo J. Math. 41 (2018), no. 2, 385--406. 	

	
\end{thebibliography}
\end{document}